\documentclass[12pt, A4]{amsart}

\usepackage{amsmath,amssymb,amsthm,mathrsfs,color}
\usepackage{bbm}
\usepackage{amsbsy}
\usepackage{tikz-cd}
\usepackage{graphicx}
\usepackage{ifpdf}
\usepackage{stmaryrd}
\usepackage{hyperref}

\newtheorem{thm}{Theorem}[section]
\newtheorem{lm}[thm]{Lemma}
\newtheorem{prop}[thm]{Proposition}
\newtheorem{cor}[thm]{Corollary}
\theoremstyle{definition}
\newtheorem{dfn}[thm]{Definition}

\newcommand{\mm}{\mathbf{m}}

\newcommand{\F}{\mathbb{F}}

\newcommand{\R}{\mathbb{R}}
\newcommand{\C}{\mathbb{C}}
\newcommand{\Z}{\mathbb{Z}}
\newcommand{\Q}{\mathbb{Q}}

\renewcommand{\O}{\mathcal{O}}

\newcommand{\Apt}{\mathcal{A}}
\newcommand{\Build}{\mathcal{B}}

\newcommand{\Lus}{{\mathscr{L}}}
\newcommand{\Fr}{\text{Fr}}

\newcommand{\EL}{\textsc{L}}

\DeclareMathOperator{\Id}{Id}
\DeclareMathOperator{\Ad}{Ad}
\DeclareMathOperator{\ad}{ad}
\DeclareMathOperator{\obs}{obs}
\DeclareMathOperator{\pr}{pr}

\DeclareMathOperator{\tame}{tame}
\DeclareMathOperator{\cond}{cond}
\DeclareMathOperator{\rel}{rel}
\DeclareMathOperator{\SC}{sc}

\DeclareMathOperator{\diag}{diag}

\DeclareMathOperator{\sep}{sep}
\DeclareMathOperator{\unr}{unr}

\DeclareMathOperator{\Zar}{Zar}

\DeclareMathOperator{\val}{val}
\DeclareMathOperator{\Gal}{Gal}

\DeclareMathOperator{\Vol}{Vol}
\DeclareMathOperator{\Hom}{Hom}

\DeclareMathOperator{\Ind}{Ind}
\DeclareMathOperator{\Res}{Res}

\DeclareMathOperator{\cind}{cind}

\DeclareMathOperator{\Cok}{Cok}
\DeclareMathOperator{\Ker}{Ker}

\DeclareMathOperator{\Zind}{Zind}

\DeclareMathOperator{\Wh}{Wh}

\newcommand{\Cat}[1]{ {\mathsf{#1}} }
\newcommand{\sch}[1]{\underline{\boldsymbol{ \mathrm{#1}}}}
\newcommand{\alg}[1]{\boldsymbol{\mathrm{#1}}}

\newcommand{\Lie}[1]{ {\mathfrak{#1}} }

\newcommand{\sheaf}[1]{{\mathscr{#1}}}

\newcommand{\OO}{\mathcal{O}}
\newcommand{\Inertia}{\mathcal{I}}
\newcommand{\Weil}{\mathcal{W}}
\newcommand{\Lang}{\mathcal{L}}

\newcommand{\msc}[1]{\mathscr{#1}}

\newcommand{\p}{\mathbf{p}}
\newcommand{\q}{\mathbf{q}}

\newcommand{\into}{\hookrightarrow}
\newcommand{\onto}{\twoheadrightarrow}
\newcommand{\From}{\colon}
\newcommand{\To}{\rightarrow}

\newcommand{\inarrow}{\arrow[hook]}
\newcommand{\onarrow}{\arrow[two heads]}

\newcommand{\set}[1]{\left\{#1\right\}}

\newcommand{\isom}{\cong}
\newcommand{\defeq}{:=}

\title[Whittaker models for covering groups]{Whittaker models for depth zero representations of covering groups}
\author[Fan Gao]{Fan Gao}
\author[Martin H. Weissman]{Martin H.~Weissman}
\address[Fan Gao]{Department of Mathematics, Purdue University, 150 N. University Street, West Lafayette, IN 47907}
\address[Martin H. Weissman]{Department of Mathematics, 1156 High Street, University of California, Santa Cruz, CA 95064}
\email[Fan Gao]{fangao.nus@gmail.com}
\email[Martin H. Weissman]{weissman@ucsc.edu}

\begin{document}

\begin{abstract}
We study the dimension of the space of Whittaker functionals for depth zero representations of covering groups. In particular, we determine such dimensions for arbitrary Brylinski-Deligne coverings of the general linear group. The results in the paper are motivated by and compatible with the work of Howard and the second author, and earlier work by Blondel.
\end{abstract}

\maketitle

\tableofcontents

%%% section %%%
\section*{Introduction}

Let $F$ be a nonarchimedean local field, and $\alg{G}$ a quasisplit reductive group over $F$. For an irreducible admissible representation $\pi$ of $G \defeq \alg{G}(F)$, the uniqueness of Whittaker functionals of $\pi$ is a powerful tool for understanding the representation. On the arithmetic side, existence (i.e., genericity) and uniqueness of Whittaker functionals enable the theory Langlands-Shahidi $L$-functions (see \cite{Sha}).  In the Langlands correspondence, generic representations serve as markers within tempered L-packets, by a conjecture of Shahidi, elucidating the fine structure of the Langlands parameterization. 

For nonlinear covering groups of $G$, the dimension of the space Whittaker functionals is more delicate as uniqueness often fails (the dimension is often greater than one).  This has been the primary obstacle to extending Langlands-Shahidi $L$-functions to covering groups in \cite{Ga1}. In a previous work, the first author \cite{Ga2} investigated genericity for some unramified theta representations of covering groups.  The existence of Whittaker functionals depends on structural and arithmetic data associated to the theta representation. 

The majority of earlier work has focused on the case of principal series and theta representations.  Here, motivated by \cite{HW} and \cite{Blo}, we consider depth zero supercuspidal representations of a covering group $\tilde G$ in the framework of Brylinski-Deligne \cite{BD}.  We expect that, as in the work of DeBacker and Reeder \cite{DR} for linear groups, the depth zero representations for covering groups become a fruitful test case for the local Langlands conjectures, in which the fine structure of L-packets can be studied in detail.  Now we give an outline and state the main results of the paper. 

\subsection*{Outline}

In \S \ref{Section-structure}, we describe the class of covering groups under consideration.  By working with covering groups over $\O$, i.e., unramified covering groups, we obtain a group $G = \alg{G}(F)$ equipped with a hyperspecial maximal compact subgroup $G_\circ = \sch{G}(\O) \subset G$, a central extension $\mu_n \into \tilde G \onto G$, and a splitting $G_\circ \into \tilde G$.  On the other hand, the extension $\tilde G$ does not come with a canonical splitting over other parahoric subgroups.  Theorem \ref{parahoric-cover} describes the resulting extensions $\mu_n \into \tilde G_x \onto G_x$ for all parahorics, in terms of the root data of a ``residual extension''.  Such results were given in the earlier work of the second author \cite{We5}, but here we hope the proof is cleaner.

In \S 2, we turn our attention to the genuine irreducible representations of $\tilde G$, and their $\psi$-Whittaker functionals.  We focus quickly on the depth zero supercuspidal representations.  These provide an excellent family of test cases, and if one believes that the work of J.K.~Yu \cite{Yu} and Ju Lee Kim \cite{Kim} extends to covering groups, then the depth zero supercuspidal representations play a role in all representations (for $p$ sufficiently large).  

After \cite{HW}, the work of Moy-Prasad \cite{MP1}, \cite{MP2}, and Morris \cite{Mor} extends to covering groups:  every genuine depth zero supercuspidal irrep $\pi$ of $\tilde G$ arises, via extension and induction, from a cuspidal representation $\rho$ of a finite reductive group.  

Based on this, Theorem \ref{OneSupp} relates the space of $\psi$-Whittaker functionals on $\pi$ to a corresponding space of Whittaker functionals for $\rho$.  The methods here are very similar to those of Blondel \cite{Blo}, who pursued similar goals.  While for linear groups, the method of Gelfand and Kazhdan \cite{GK} (see also \cite{Rod}) gives uniqueness of Whittaker models, their method does not adapt easily to general covering groups.  As a replacement, we bootstrap the uniqueness of Whittaker models for groups over finite fields, using the fact that {\em nonlinear covers} of parahoric subgroups are related to {\em linear} groups over finite fields (the idea of \cite[Construction 12.11]{BD}).

The remainder of \S 2 refines estimates for the space of $\psi$-Whittaker functionals on $\pi$, relating its dimension to the index of sublattices within $Y^{W \rtimes \Fr}$ of Weyl- and Frobenius-invariant cocharacters of a maximal torus in $\alg{G}$.  The sharp result is Theorem \ref{Bound1}, that $\dim \Wh_\psi(\pi) = [Y^{W \rtimes \Fr} : Y_{x,\rho}]$  and the rest of the section describes the lattice $Y_{x,\rho}$.  An easy consequence is the uniqueness of Whittaker models, $\dim \Wh_\psi(\pi) = 1$ when $\alg{G}$ is {\em semisimple}, and $\pi$ is a genuine, depth zero, generic, supercuspidal irrep of $\tilde G$.  Such a result, for covers of $\alg{SL}_n$, was obtained in a different fashion in the PhD thesis of Stephen Devlin \cite{Dev}.

The depth zero supercuspidal representations of $\tilde G$ arise from cuspidal representations of a finite reductive group $\alg{\bar G}'(\F_q)$.  These are the subject of the work of Deligne and Lusztig in \cite{DL} (and earlier work of Green \cite{Gre} for the general linear group).  In the generic case (at least for groups with connected center), such cuspidal representations of $\alg{\bar G}'(\F_q)$ arise from characters of minisotropic tori, in general position.  

In \S 3, we review the construction of Deligne-Lusztig representations, and relate the lattice $Y_{x,\rho}$ to the Deligne-Lusztig data $({}^w \alg{\bar T}', \theta')$ (a minisotropic torus and character).  In \cite{DL}, Deligne and Lusztig parameterize such data by semisimple conjugacy classes in a dual group; this has been rephrased (with fewer choices) in terms of a complex dual group by Lusztig in \cite{Lus}, and this seems more relevant for the local Langlands correspondence, especially after DeBacker and Reeder \cite{DR}.  In this spirit, we provide a description of $Y_{x,\rho}$ in terms of the Lusztig parameter, in Proposition \ref{thetaProp}.  In this way, we effectively describe the dimension $\dim \Wh_\psi(\pi)$ in terms of the Lusztig parameter of the cuspidal representation used to construct $\pi$.

Finally, in \S 4, we apply these results to all covers of $\alg{GL}_r$.  Some of these covers have been studied by Kazhdan and Patterson in \cite{KP}, but the class of Brylinski-Deligne covers is broader.  By passing through the Lusztig parameters as above, we describe the dimensions $\dim \Wh_\psi(\pi)$ for $\pi$ a genuine, generic, depth zero, supercuspidal irrep of a cover of $GL_r$.  This recovers results of \cite{Blo} in the case of Kazhdan-Patterson coverings -- indeed the basic methods are not so different from \cite{Blo}.  But by passing through Lusztig parameters, some computations are made simpler and generalize easily.  As Lusztig parameters should be connected to a local Langlands parameterization for covering groups, this suggests that the mysterious dimensions $\dim \Wh_\psi(\pi)$ are predictable from the Langlands parameter.  %We finish with speculations along these lines.  To go beyond speculation, it seems necessary to begin a study of pure (or rigid) inner forms of covering groups.  

We remark that for linear algebraic groups,  the dimension of certain degenerate Whittaker functionals is obtained by Moeglin-Waldspurger \cite{MW}, in terms of the Harish-Chandra-Howe local character expansion. For covering groups, W.-W.~Li \cite{Li} has shown that the character expansion holds, and furthermore the results in \cite{MW} are extended to covering groups in \cite{Pat}. Therefore, it would be interesting to compute explicitly the local character expansion of the depth zero supercuspidal representation in our paper and compare the results.

As it is precarious (and sometimes incorrect) to claim that results for linear algebraic groups extend to covering groups evidently, or even from one class of covering groups to another, we provide details whenever possible. Only if the argument involves simple cosmetic change from existing ones in the literature, we omit details and are content with giving the reference or outline of the proof.  An example is the ``heredity theorem'' of Rodier \cite{Rod}, extended to Kazhdan-Patterson covering groups by Banks \cite{Ban}, which extends with only cosmetic changes to Brylinski-Deligne covering groups.  See \S \ref{Rodier} for details.

This paper does not introduce significant new tools in representation theory.  But by working in a very general setting of covering groups, and by connecting Whittaker models to Lusztig parameters, the dimensions $\dim \Wh_\psi(\pi)$ should be easier to compute.  In particular, one can easily predict the uniqueness of Whittaker models, when it occurs.  Eventually we expect the results here to shed light on L-packets for coverings of reductive groups over local fields.  

\subsection*{Acknowledgments} 
The first author would like to thank Corinne Blondel for several interesting comments on her earlier work on the topic. The second author would like to thank Freydoon Shahidi and the first author for insights and hospitality during a visit to Purdue University.  This work was supported by a grant from the Simons Foundation (\#426453, Martin Weissman).
%%% section %%%

\section{Structure theory}
\label{Section-structure}

We begin by recalling the structure theory of unramified covering groups from \cite{We3}.  Let $F$ be a nonarchimedean local field, with ring of integers $\O$ and residue field $\F_q$ of order $q$.  Write $\varpi$ for a uniformizer of $F$, and $\val \From F \To \Z$ a discrete valuation such that $\val(\varpi) = 1$.  Let $F^{\sep}$ be a separable algebraic closure of $F$, and $F^{\unr}$ the maximal unramified extension of $F$ in $F^{\sep}$.  Let $\O^{\unr}$ be the integral closure of $\O$ in $F^{\unr}$.  The residue field of $\O^{\unr}$ is an algebraic closure $\bar \F_q / \F_q$.  Let $\Fr$ be the geometric Frobenius automorphism, viewed as an automorphism of $F^{\unr}$ or of $\bar \F_q$.

\subsection{The unramified group}

Let $\sch{G}$ be a reductive group over $\O$, i.e., a smooth group scheme over $\O$ with connected reductive geometric fibres.  Then $\alg{G} = \sch{G}_F$ is quasisplit and splits over $F^{\unr}$.  Let $G_\circ = \sch{G}(\O)$ and $G = \alg{G}(F)$; then $G_\circ$ is a hyperspecial maximal compact subgroup of $G$.  Write $G^{\unr} = \alg{G}(F^{\unr})$ and $G_\circ^{\unr} = \alg{G}(\O^{\unr})$.  As for $\alg{G}$, write 
$$X = \sch{X}(F), \quad X_\circ = \sch{X}(\O), \quad X^{\unr} = \sch{X}(F^{\unr}), \quad X_\circ^{\unr} = \sch{X}(\O^{\unr}),$$ 
for any $\O$-scheme $\sch{X}$.

Define $\alg{\bar G} = \sch{\bar G}_{\F_q}$, the special fibre of $\sch{G}$.  As a variety over a field, we adopt the classical perspective and also write $\alg{\bar G}$ for its set of $\bar \F_q$-points:  $\alg{\bar G} = \alg{\bar G}(\bar \F_q)$.  We write $\bar G = \alg{\bar G}(\F_q) = \alg{\bar G}^{\Fr}$ for the $\F_q$-points.  Similarly for any (smooth) scheme $\sch{X}$ over $\O$, we write
$$\alg{\bar X} = \sch{X}_{\F_q}, \quad \alg{\bar X} = \alg{\bar X}(\bar \F_q), \quad \bar X = \alg{\bar X}^{\Fr}.$$

Let $\sch{S}$ be a maximal split torus in $\sch{G}$ over $\O$, let $\sch{T}$ be its centralizer, a maximal torus in $\sch{G}$ defined over $\O$.   Let $\sch{B}$ be a Borel subgroup of $\sch{G}$ defined over $\O$ and containing $\sch{T}$, and let $\sch{U}$ be the unipotent radical of $\sch{B}$.  Let $\sch{N}$ be the normalizer of $\sch{T}$ in $\sch{G}$. 

Let $(X, \Phi, \Delta, Y, \Phi^\vee, \Delta^\vee)$ be the based root datum of $\sch{G}$ with respect to $\sch{B}$ and $\sch{T}$.  Here $X$ and $Y$ are the character and cocharacter lattices of $\sch{T}$, $\Phi$ and $\Phi^\vee$ are the sets of (absolute) roots and coroots, and $\Delta$ and $\Delta^\vee$ the simple roots and simple coroots with respect to the choice of Borel subgroup $\sch{B}$.  Define $Y^{\SC} \subset Y$ be the sublattice generated by $\Phi^\vee$.  The root datum is endowed with an action of $\Fr$.

If $\alpha \in \Phi$, we write $\sch{U}_\alpha$ for the corresponding root subgroup of $(\sch{G})_{\O^{\unr}}$.  Thus $\sch{U}_\alpha$ is a group scheme over $\O^{\unr}$, isomorphic to the additive group scheme $\sch{G}_a$.  For any root $\alpha \in \Phi$, it is possible to choose a pair of isomorphisms (over $\O^{\unr}$)
$$e_\alpha \From \sch{G}_a \To \sch{U}_\alpha, \quad e_{-\alpha} \From \sch{G}_a \To \sch{U}_{-\alpha},$$
such that the following is a map of schemes $n_\alpha \From \sch{G}_m \To \sch{N}$ over $\O^{\unr}$:
$$n_\alpha(t) \defeq e_\alpha(t) e_{-\alpha}(-t^{-1}) e_\alpha(t).$$
When the $e_\alpha, e_{-\alpha}$ are so chosen, we also define $h_\alpha \From \sch{G}_m \To \sch{T}$ by
$$h_\alpha(t) \defeq n_\alpha(t) n_\alpha(-1).$$
It turns out that $h_\alpha$ is a homomorphism of group schemes over $\O^{\unr}$; it is precisely the coroot $\alpha^\vee$.

\subsection{The full tame extension}

Let $\sch{G}'$ be a central extension of $\sch{G}$ by $\alg{K}_2$, in the category of sheaves of groups on $\O_{\Zar}$.  Such extensions are the focus of the work of Brylinski-Deligne \cite{BD}, who carried out a classification when working over a base field; over $\O_{\Zar}$, the classification of \cite{BD} has been adapted in \cite{We3}.  Extending scalars from $\O$ to $F$, we have a central extension of sheaves of groups on $F_{\Zar}$,
\begin{equation}
\label{cextG}
\alg{K}_2 \into \alg{G}' \onto \alg{G}.
\end{equation}

The extension $\sch{G}'$ is classified by three invariants, denoted $(Q, \sheaf{D}, f)$.  The first is a $\Fr$- and Weyl-invariant quadratic form $Q \From Y \to \Z$.  Let $B_Q \From Y \otimes Y \to \Z$ be the associated bilinear form,
$$B_Q(y_1, y_2) =Q(y_1+y_2)-Q(y_1) -Q(y_2).$$

Taking $F$-points of \eqref{cextG} yields a central extension $\alg{K}_2(F) \into \alg{G}'(F) \onto G$.  The tame (residue) symbol in K-theory gives a surjective homomorphism,
$$\partial \From \alg{K}_2(F) \onto \F_q^\times; \quad  \partial \{ u,v \} = (-1)^{\val(u) \val(v)} \cdot \pr \left( \frac{v^{\val(u)}}{u^{\val(v)}} \right).$$
Here $\pr \From \O \To \F_q$ is the reduction map.  Pushing out the extension via $\partial$ yields what we call the {\em full tame extension}
$$\F_q^\times \into G' \onto G.$$
Since the extension $\sch{G}'$ is defined over $\O$, this extension comes with a canonical splitting over $G_\circ$.  Write $\iota_\circ$ for this canonical splitting, $\iota \From G_\circ \into G'$.  We describe $\iota_\circ$ further in a subsequent section.

This construction behaves well for unramified extensions; the extension $\F_q^\times \into G' \onto G$ arises also as the $\Fr$-fixed points of the extension obtained from taking $F^{\unr}$-points of \eqref{cextG} and applying the tame symbol there,
\begin{equation}
\label{fulltame}
\bar \F_q^\times \into (G^{\unr})' \onto G^{\unr}.
\end{equation}
This extension comes with a canonical splitting over $G_\circ^{\unr} = \sch{G}(\O^{\unr})$.  

%Pushing out further, via the $(q-1)/n$ power map $\F_q^\times \onto \mu_n$ yields a topological central extension
%$$\mu_n \into \tilde G \onto G.$$
%We view $G_\circ$ as a subgroup of $\tilde G$.

%If $H \subseteq G$ is any subset, we write $\tilde H$ for the preimage of $H$ in $\tilde G$.  If $\tilde g_1, \tilde g_2 \in \tilde G$, write $[\tilde g_1, \tilde g_2] \defeq \tilde g_1 \tilde g_2 \tilde g_1^{-1} \tilde g_2^{-1}$ for their commutator.
 
%The central extension $\tilde G$ splits canonically over the $F$-points of any unipotent $F$-subgroup of $\alg{G}$.  Hence, if $u$ is a unipotent element of $G$, there is a canonical element $\tilde u \in \tilde G$ lifting $u$, and the map $u \mapsto \tilde u$ is equivariant for the conjugation-action of $G$ on its unipotent elements.

\subsection{Parahoric subgroups}

Let $\Build^{\unr}$ denote the reduced building of $\alg{G}_{F^{\unr}}$, i.e., the building of the adjoint quotient of $\alg{G}$ over $F^{\unr}$.  Let $\Apt^{\unr}$ be the apartment in $\Build^{\unr}$ corresponding to the maximal torus $\alg{T}$ (which splits over $F^{\unr}$).  Let $\Build$ be the building of $\alg{G}$ over $F$, naturally identified with the fixed points $(\Build^{\unr})^{\Fr}$; the apartment $\Apt \subset \Build$ associated to the maximal $F$-split torus $\alg{S}$ can be identified with $(\Apt^{\unr})^{\Fr}$.  

If $x \in \Build$, then we write $G_x$ for the corresponding parahoric subgroup of $G$ and $G_x^+$ for its pro-unipotent radical.  Bruhat and Tits give a canonical smooth group scheme $\sch{G}_x$ over $\O$, for which $\sch{G}_x(\O) = G_x$.  We follow the convention for parahorics in which the special fibre $\alg{\bar G}_x = (\sch{G}_x)_{\F_q}$ is a connected affine algebraic group.

We write $\circ$ for the base point in $\Apt$ at which $G_\circ = \sch{G}_\circ(\O) = \sch{G}(\O)$.  Placing $\circ$ at zero, we may identify $\Apt$ with $(Y \otimes_\Z \R)^{\Fr}$.  If $x \in \Apt$, then we have inclusions of group schemes over $\O$,
$$\sch{S} \subset \sch{T} \subset \sch{G}_x.$$ 
In this case, write $\alg{\bar M}_x$ for the unique standard Levi component of $\alg{\bar G}_x$ -- one containing $\alg{\bar T}$.  Then $\bar M_x = \alg{\bar M}_x(\F_q)$ is identified with $G_x / G_x^+$.

The root datum of $\alg{\bar M}_x$ with respect to $\alg{\bar T}$ can be described as follows:  let $\Phi_x$ be the set of roots $\alpha \in \Phi$ such that $\alpha(x)$ is an integer (viewing $x \in Y \otimes_\Z \R$); let $\Phi_x^\vee$ be the set of corresponding coroots.  The original choice of Borel subgroup defines a subset $\Phi_x^+$ of positive roots, and a set of simple positive roots $\Delta_x$ within $\Phi_x^+$.  The root datum of $\alg{\bar M}_x$  is the sextuple (endowed with $\Fr$-action),
$$(X, \Phi_x, \Delta_x, Y, \Phi_x^\vee, \Delta_x^\vee).$$

\subsection{Covers of parahorics}

Here we consider a point $x \in \Apt$, the parahoric subgroup $G_x$, and the full tame extension restricted to $G_x$,
\begin{equation}
\label{coverGx}
\F_q \into G_x' \onto G_x.
\end{equation}
Recall that $\sch{G}_x$ is a smooth group scheme over $\O$ with special fibre $\alg{\bar G}_x$ and general fibre $\alg{G}$.  In such a context, \cite[Construction 12.11]{BD} gives a functor,
$$\partial \From \Cat{CExt}(\alg{G}, \alg{K}_2) \To \Cat{CExt}(\alg{\bar G}_x, \alg{\bar G}_m).$$
Applied to $\alg{K}_2 \into \alg{G}' \onto \alg{G}$, this functor gives a central extension of {\em affine algebraic groups} over $\F_q$,
$$\alg{\bar G}_m \into \alg{\bar G}_x' \onto \alg{\bar G}_x.$$
Such an extension splits uniquely over the unipotent radical $\alg{\bar U}_x$ of $\alg{\bar G}_x$.  Thus it arises as the pullback of a central extension of the standard Levi subgroup,
$$\alg{\bar G}_m \into \alg{\bar M}_x' \onto \alg{\bar M}_x.$$
We say that $\alg{\bar M}_x'$ is the {\em residual extension} at $x$.  Taking $\F_q$-points yields a central extension,
\begin{equation}
\label{residMx}
\F_q^\times \into \bar M_x' \onto \bar M_x.
\end{equation}

The results of \cite[\S 12.8-12.2]{BD} connect the residual extension \eqref{residMx} to the cover of the parahoric in \eqref{coverGx}.  The full tame extension of the parahoric is the pullback of the residual extension, by the projection map $G_x \onto G_x / G_x^+ = \bar M_x$.
\begin{equation}
\label{respullback}
\begin{tikzcd}
\F_q^\times \inarrow{r}  \arrow{d}{=} & G_x' \onarrow{r} \arrow{d}{\pr} & G_x \onarrow{d}{\pr = \text{projection mod } G_x^+} \\
\F_q^\times \arrow{r} & \bar M_x' \arrow{r} & \bar M_x
\end{tikzcd}
\end{equation}
Thus to know the extension $G_x'$, it suffices to know the residual extension $\alg{\bar M}_x'$ as a central extension of $\alg{\bar M}_x$ by $\alg{\bar G}_m$ over $\F_q$. 

\subsection{Description of the residual extension}
\label{ResSection}

Central extensions of reductive groups by the multiplicative group are classified in \cite{We3}; we review this classification here.  Recall that $\alg{\bar T}$ is a maximal torus in $\alg{\bar M}_x$, and $Y$ is the cocharacter lattice of $\alg{\bar T}$.  The central extension $\alg{\bar G}_m \into \alg{\bar M}_x' \onto \alg{\bar M}_x$ restricts to an extension of tori $\alg{\bar G}_m \into \alg{\bar T}' \onto \alg{\bar T}$, which gives an extension of $\Z[\Fr]$-modules
\begin{equation}
\label{extY}
\Z \into Y' \onto Y.
\end{equation}

Write $\alg{\bar M}_x^{\SC}$ for the simply-connected cover of the derived subgroup of $\alg{\bar M}_x$.  Let $Y_x^{\SC}$ be the $\Z$-span of $\Phi_x^\vee$ in $Y$.  Pulling back via $\alg{\bar M}_x^{\SC} \To \alg{\bar M}_x$, the central extension
$$\alg{\bar G}_m \into (\alg{\bar M}_x^{\SC})' \onto \alg{\bar M}_x^{\SC}$$
splits uniquely.  This gives a canonical splitting $\iota_x \From Y_x^{\SC} \into Y'$ of \eqref{extY} over $Y_x^{\SC}$.  

The pair $(Y', Y_x^{\SC} \into Y')$ determines the residual extension up to unique isomorphism.  More precisely, \cite{We3} demonstrates that the Picard category $\Cat{CExt}(\alg{\bar M}_x, \alg{\bar G}_m)$ is equivalent (via the construction above) to the category of such pairs $(Y', \iota_x \From Y_x^{\SC} \into Y')$, where $Y'$ is an extension of $Y$ by $\Z$ (as $\Z[\Fr]$-modules), and $Y_x^{\SC} \into Y'$ is a splitting of the extension over $Y_x^{\SC}$.  

\begin{prop}
Suppose that $\alg{\bar M}_x$ is split and the derived subgroup of $\alg{\bar M}_x$ is simply-connected.  Then the central extension $\alg{\bar G}_m \into \alg{\bar M}_x' \onto \alg{\bar M}_x$ splits.
\end{prop}
\begin{proof}
If the derived subgroup of $\alg{\bar M}_x$ is simply-connected, then $Y / Y_x^{\SC}$ is free.  It follows that any splitting $\iota_x \From Y_x^{\SC} \into Y'$ extends to a splitting $Y \into Y'$.  Thus the pair $(Y', \iota_x)$ is trivial, and so the central extension is also trivial.  
\end{proof}   

To describe the residual extension in general, it remains to describe $Y'$ and $\iota_x \From Y_x^{\SC} \into Y'$.
\begin{thm}
\label{parahoric-cover}
The hyperspecial point $\circ$ identifies $Y' = Y \oplus \Z$, and $\iota_x \From Y_x^{\SC} \into Y'$ with the unique homomorphism satisfying
$$\iota_x(\alpha^\vee) = \left( \alpha^\vee, \alpha(x) \cdot Q(\alpha^\vee) \right) \text{ for all } \alpha \in \Phi_x.$$
\end{thm}
\begin{proof}
At the hyperspecial base point $\circ$, we find the pair $(Y', \iota_\circ \From Y_\circ^{\SC} \into Y')$.  The extension $\Z \into Y' \onto Y$ is the same at $\circ$ as at $x$, since the cover of the torus $\alg{T}$ is the same everywhere in the apartment.  Since $\circ$ is a hyperspecial point, $Y_\circ^{\SC} = Y^{\SC}$ is the span of the coroots in $Y$.  But since the cover $\sch{G}'$ is defined over $\O$, the residual extension at $\circ$  
$$\alg{\bar G}_m \into \alg{\bar G}' \onto \alg{\bar G}$$
splits canonically.  (See \cite[\S 3]{We3} for details).  Thus $(Y', Y_\circ^{\SC} \into Y')$ is canonically isomorphic to the pair $(Y \oplus \Z, \iota_\circ)$ where $\iota_\circ(\alpha^\vee) = (\alpha^\vee, 0)$ for every coroot $\alpha^\vee$.

To determine $\iota_x \From Y_x^{\SC} \into Y'$, we have to work a bit harder.  Given $\alpha \in \Phi_x$, choose $e_\alpha, e_{-\alpha}$ as before.  Noting that the formula in the theorem is $\Fr$-invariant, we may prove the result by passing to an unramified extension of $\O$ over which $\sch{G}$ is split.  Thus we assume without loss of generality that $\sch{G}$ is split over $\O$, and thus $e_\alpha, e_{-\alpha}$ are defined over $\O$ in what follows.

The central extension $\sch{K}_2 \into \sch{G}' \onto \sch{G}$ splits uniquely over every unipotent subgroup.  Splitting over the root subgroups $\alg{U}_\alpha$, we find canonical lifts of $e_\alpha, e_{-\alpha}$,
$$e_\alpha' \From \sch{G}_a \To \sch{U}_\alpha', \quad e_{-\alpha}' \From \sch{G}_a \To \sch{U}_{-\alpha}'.$$
Define lifts $n_\alpha'$ and $h_\alpha'$ analogously to $n_\alpha, h_\alpha$, by replacing $e_\alpha$ by $e_\alpha'$ and $e_{-\alpha}$ by $e_{-\alpha}'$ wherever they occur.  In the full tame extension, we find maps
$$e_\alpha' \From F \To U_\alpha', \quad n_\alpha' \From F^\times \To N', \quad h_\alpha' \From F^\times \To T'$$
Here $U_\alpha'$, $N'$, and $T'$ are the preimages of $U_\alpha$, $N$, and $T$ in the full tame extension $\F_q^\times \into G' \onto G$.  

While the map $h_\alpha \From F^\times \To T$ was a homomorphism, the lift $h_\alpha'$ is not necessarily a homomorphism.  But the computations of Brylinski and Deligne (see \cite[\S 11.1.5]{BD}) give a formula,
\begin{equation}
\label{hcocyc}
h_\alpha'(u) \cdot h_\alpha'(v) = \partial \{ u,v \}^{Q(\alpha^\vee)} \cdot h_\alpha'(uv).
\end{equation}

Returning to the parahoric $G_x$, if $t \in \O^\times$ and $k = \alpha(x)$, then $-k = -\alpha(x)$ and 
$$e_\alpha'(t \varpi^k) \in G_x', \quad e_{-\alpha}'(- t^{-1} \varpi^{-k}) \in G_x'.$$
Projection $\pr \From G_x' \onto \bar M_x'$ yields well-defined elements of $\bar M_x'$ for all $t \in \bar \F_q^\times$,
$$\bar e_{x, \alpha}'(\bar t) = \pr \left( e_\alpha'(t \varpi^k) \right) \in \bar M_x', \quad \bar e_{x, -\alpha}'(-\bar t^{-1}) = \pr \left( e_\alpha'(-t^{-1} \varpi^{-k}) \right) \in \bar M_x' .$$

Define, by adapting previous definitions,
$$\bar n_{x,\alpha}'(\bar t) = \bar e_{x, \alpha}'( \bar t) \bar e_{x, -\alpha}'(-\bar t^{-1}) \bar e_{x, \alpha}'(\bar t) \in \bar N_x',$$
$$\bar h_{x, \alpha}'(\bar t) = \bar n_{x,\alpha}'(\bar t) \bar n_{x, \alpha}'(-1) \in \bar T'.$$
Here $\bar N_x'$ denotes the normalizer of $\bar T$ in $\bar M_x'$.

The map $\bar h_{x, \alpha'}$ defines a {\em homomorphism},
$$\bar h_{x, \alpha} \From \F_q^\times \To \bar M_x',$$
The residual extension splits uniquely when pulled back to $\alg{\bar M}_x^{\SC}$, and so it must be compatible with $\bar e_{x, \alpha}'$ on root subgroups (identifying root subgroups of $\alg{\bar M}_x^{\SC}$ with those of $\alg{\bar M}_x$).  It follows that 
$$\bar h_{x, \alpha}'(\bar t) = \iota_x(\alpha^\vee) (\bar t) \text{ for all } \bar t \in \F_q^\times.$$

Explicitly, using the general identity $n_\alpha'(t) \cdot n_\alpha'(-t) = 1$,
\begin{align*}
\bar h_{x, \alpha}'(\bar t) &= \pr \left( n_\alpha'(t \varpi^k) n_\alpha'(- \varpi^k) \right),  \\
&= \pr \left( n_\alpha'(t \varpi^k) n_\alpha'(-1) n_\alpha'(1) n_\alpha'(- \varpi^k) \right) \text{ since } n_\alpha'(-1) n_\alpha'(1) = 1, \\
&= \pr \left( h_\alpha'(t \varpi^k) \cdot h_\alpha'(\varpi^{k})^{-1} \right), \\
&= \pr \left( h_\alpha'(t) \cdot \partial \{ t, \varpi \}^{k Q(\alpha^\vee)} \right) \text{ by \eqref{hcocyc}}, \\
&= \bar h_\alpha(\bar t) \cdot (\bar t)^{k Q(\alpha^\vee)}.
\end{align*}
But $\bar h_\alpha(\bar t) = \iota_\circ(\alpha^\vee)(\bar t)$.  Recalling that $k = \alpha(x)$, it follows that
$$\iota_x(\alpha^\vee) = \left( \alpha^\vee, \alpha(x) \cdot Q(\alpha^\vee) \right).$$
\end{proof}

\begin{cor}
If $\alg{G}$ has simply-connected derived subgroup, then at every point $x$ in the building, the residual extension $\alg{\bar M}_x'$ has simply-connected derived subgroup.
\end{cor}
\begin{proof}
The coroot lattice for $\alg{\bar M}_x'$ is the lattice $\Lambda_x$ spanned by the set $\{ (\alpha^\vee, \alpha(x) \cdot Q(\alpha^\vee) ) : \alpha \in \Phi_x \}$.  If $y' = (y, t) \in Y' = Y \oplus \Z$ and $k \cdot y' \in \Lambda_x$ for some $k > 0$, then we would have 
$$ky \in Y_x^{\SC} = \text{Span} \{ \alpha^\vee : \alpha \in \Phi_x \}.$$  
Since $\alg{G}$ has simply-connected derived subgroup, $\alg{\bar M}_x$ has simply-connected derived subgroup, so $Y / Y_x^{\SC}$ is free and $Y_x^{\SC}$ is saturated.  Hence $ky \in Y_x^{\SC}$ implies $y \in Y_x^{\SC}$.

It follows that $ky' = (ky, kt)$ for some $y \in Y_x^{\SC}$.  Also, note that $\iota_x(ky) = (ky, ku)$ for some $u \in \Z$.  Since $k y' \in \Lambda_x$ and $\iota_x(y) \in \Lambda_x$, we find that the difference,
$$(0, kt - ku) \in \Lambda_x.$$
But this implies $kt - ku = 0$, so $t = u$, and $y' = \iota_x(y) \in \Lambda_x$.  We have proven that $k y' \in \Lambda_x$ implies $y' \in \Lambda_x$, so $\Lambda_x$ is saturated.  Thus the derived subgroup of $\alg{\bar M}_x'$ is simply-connected.
\end{proof}
 
\section{Generic genuine representations}

Recall that $\sch{S}$ is a maximal split torus in $\sch{G}$ over $\O$, and $\sch{T}$ is its centralizer.  In particular, $\sch{T}$ splits over an unramified extension of $F$, and its cocharacter lattice $Y$ becomes a $\Z[\Fr]$-module.  $\sch{B} = \sch{T} \sch{U}$ is a Borel subgroup of $\sch{G}$, defined over $\O$.  Fix a central extension $\alg{K}_2 \into \sch{G}' \onto \sch{G}$ over $\O$ in what follows.  

Write $\mu_n$ for the group of $n\textsuperscript{th}$ roots of unity in $F^\times$.  As in \cite{We4}, we assume $n$ divides $q-1$.  We identify $\mu_n$ with the group of $n\textsuperscript{th}$ roots of unity in $\F_q^\times$ when convenient (via the Teichm\"uller lift).  The full tame extension is a central extension, $\F_q^\times \into G' \onto G$.  If we push out $G'$ via the $(q-1)/n$ power map $\F_q^\times \onto \mu_n$, we obtain a topological central extension
\begin{equation}
\label{tildeG}
\mu_n \into \tilde G \onto G.
\end{equation}

In what follows, we study only {\em smooth} representations of $\tilde G$ (a locally compact, totally disconnected group) on {\em complex} vector spaces.  Fix an injective character $\epsilon \From \mu_n \into \C^\times$ throughout.  If $\tilde H$ is a subgroup of $\tilde G$ containing $\mu_n$, a representation of $\tilde H$ is called {\em genuine} if the central $\mu_n$ acts by the character $\epsilon$.  We begin by recalling the genuine representation theory of covers of tori.

\subsection{Covers of tori}

The pullback of $\tilde G$ gives a cover of $T$, which is a group of Heisenberg type,
$$\mu_n \into \tilde T \onto T.$$
While $T = \alg{T}(F)$ is an abelian group, the extension $\mu_n \into \tilde T \onto T$ is often nonabelian.  (In the classical case $\tilde G = Mp_{2n}(F)$, the extension $\tilde T$ is abelian, which explains many of the similarities between the metaplectic group and linear groups.)  Its center $Z(\tilde T)$ has finite index in $\tilde T$, and has been characterized in \cite{We1} and \cite{We2}.  As a group of Heisenberg type, the irreducible genuine representations of $\tilde T$ are in bijection with the genuine characters of its center (by taking central character).  

In characterizing the center, a special role is played by the sublattice 
$$Y_{Q,n} \defeq \set{y\in Y: B_Q(y, y')\in n\Z \text{ for all } y' \in Y} \subset Y.$$

  By \cite[Corollary 3.4]{We2}, every irreducible genuine representation of $\tilde T$  has the same (finite) dimension $\Zind(\tilde T)$, and
\begin{equation} \label{ToriH}
\Zind(\tilde T) = \sqrt{ \# (\tilde T / Z(\tilde T)) } = \# (Y^\Fr / Y_{Q,n}^\Fr).
\end{equation}
We call the positive integer $\Zind(\tilde T)$ the {\em central index} of $\tilde T$.

%$[Y: Y_{Q,n}]=[\ol{S}: Z(\ol{S})\cdot \mathbf{S}(\O)]$. Denote by $\mbf{S}_{Q,n}$ the split torus over $F$ corresponding to $Y_{Q,n}$, and write $S_{Q,n}:=\mbf{S}_{Q,n}(F)$. Let $i_{Q,n}: S_{Q,n} \to S$ be the isogeny induced by the embedding $Y_{Q,n} \into Y$, then we have
%\begin{equation} \label{cent}
%Z(\ol{S})=\wp^{-1}\big(i_{Q,n}(S_{Q,n}) \big).
%\end{equation}
%The covering torus $\ol{S}$ associated with the quadratic form $Q$ is a Heisenberg group. 

% I think this section should be partially rewritten, adapting the methods of DeBacker-Reeder to covering groups rather than assuming a splitting and referring to D-R.  Otherwise there will be trouble in cases where there are multiple G-orbits of hyperspecial vertices.
%  But use D-R and splittings to relate dimensions of spaces of Whittaker models.
%  Use existence of integral models of covering groups at varying hyperspecial verts -- action of adjoint group should do it?
 
\subsection{Whittaker functionals}

Since central extensions of reductive groups by $\alg{K}_2$ split uniquely over unipotent subgroups, we may view $U$ as a subgroup of $\tilde G$, and $\tilde B = \tilde T \ltimes U$.  The conjugation action of $\tilde T$ on $U$ factors through $T$.  By this conjugation action, $\tilde T$ or $T$ acts on the space of characters $\psi \From U \To \C^\times$; define ${}^t \psi(u) = \psi(t^{-1} u t)$.  

We follow Casselman and Shalika \cite[\S 3]{CS} to decompose characters $\psi \From U \To \C^\times$ along root subgroups.  Let $\Phi_{\rel}$ be the set of relative roots, i.e., roots for the adjoint action of the maximal split torus $\alg{S}$ on the Lie algebra $\Lie{g}$ of $\alg{G}$.  For any relative root $\beta \in \Phi_{\rel}$, one may define a root subgroup $\alg{U}_\beta$, defined over $F$ (in fact, over $\OO$).  It is a unipotent group, whose Lie algebra is $\Lie{g}_\beta \oplus \Lie{g}_{2 \beta}$ (the latter occurring if and only if $2 \beta \in \Phi_{\rel}$).  

Let $\alg{S}_\beta$ be the neutral component of the kernel of $\beta$, and let $\alg{G}_\beta$ be the derived subgroup of the centralizer of $\alg{S}_\beta$ in $\alg{G}$.   Then $\alg{G}_\beta$ is a semisimple unramified group of $F$-rank $1$, containing the root subgroups $\alg{U}_{\pm \beta}$.  Depending on whether $2 \beta \in \Phi_{\rel}$ or $2 \beta \not \in \Phi_{\rel}$, there is an isogeny
$$f_\beta \From \alg{SU}_3^{F_\beta} \onto \alg{G}_\beta \text{ or } f_\beta \From \alg{SL}_2^{F_\beta} \onto \alg{G}_\beta,$$
where $F_\beta \subset F^{\unr}$ is a finite unramified extension of $F$, and the groups above denote the restriction of scalars from $F_\beta$ down to $F$ of a quasisplit unramified unitary group $\alg{SU}_3$ or $\alg{SL}_2$, respectively.  

Write $E_\beta$ for the quadratic unramified extension of $F_\beta$ in $F^{\unr}$ if $2 \beta \in \Phi_{\rel}$, or write $E_\beta = F_\beta$ otherwise.  With this notation, the isogeny $f_\beta$ determines an isomorphism in both cases,
$$f_\beta \From U_\beta / U_{2 \beta} \xrightarrow{\sim} E_\beta.$$
Such isomorphisms are not determined ``on the nose'' by our original data $\sch{G}, \sch{B}, \sch{T}$.  But they are determined up to scaling by $\OO_{E_\beta}^\times$, as $U_\beta / U_{2 \beta}$ has a distinguished compact open subgroup obtained by intersecting with $G_\circ = \sch{G}(\OO)$.
In particular, there is a well-defined valuation,
$$\val_\beta \From U_\beta / U_{2 \beta} \To \Z \cup \{ \infty \},$$
obtained from the valuation on $E_\beta$ via $f_\beta$.  Here the valuation is normalized so that $\val(F^\times) = \val((F^{\unr})^\times) = \Z$.

The $F$-Borel subgroup $\alg{B}$ determines a subset $\Phi_{\rel}^+ \subset \Phi_{\rel}$ of positive relative roots, and simple relative roots $\Delta_{\rel} \subset \Phi_{\rel}^+$.  Then there is natural isomorphism,
$$\alg{U} / [ \alg{U}, \alg{U} ] \isom \prod_{\beta \in \Delta_{\rel}} \alg{U}_\beta / \alg{U}_{2 \beta}.$$
In this way, every character $\psi \From U \To \C^\times$ can be uniquely decomposed:
$$\psi = \prod_{\beta \in \Delta_{\rel}} \psi_\beta, \quad \psi_\beta \From U_\beta / U_{2 \beta} \To \C^\times.$$

\begin{dfn}
The character $\psi \From U \To \C^\times$ is called {\em generic} if $\psi_\beta$ is nontrivial for every $\beta \in \Delta_{\rel}$.  (Rodier \cite[\S II.2]{Rod2} and Casselman-Shalika \cite[\S 3]{CS} and others use the term {\em principal}.)
\end{dfn}

\begin{dfn}
The {\em conductor} of a generic character $\psi \From U \To \C^\times$ is the family of integers $\{ \cond_\beta(\psi) : \beta \in \Delta_{\rel} \}$, where 
$$\cond_\beta(\psi) = \min \{ n \in \Z : u \in U_\beta / U_{2 \beta} \text{ and } \val(u) \geq n \Longrightarrow \psi_\beta(u) = 1\}.$$
\end{dfn}

%% Perhaps insert a remark or paragraph, describing how the conductor can be thought of as an element of Y_{\ad}^{\Fr}.

For what follows, it is helpful to know how the adjoint action of $T$ affects the conductor of a generic character.  Note that since $\alg{G}$ is quasisplit, and splits over $F^{\unr}$, the set of relative simple roots $\Delta_{\rel}$ can be identified with the set of $\Fr$-orbits on $\Delta$.  The identification is given simply by restricting a character of $\alg{T}$ to one of $\alg{S}$.
\begin{prop}
If $t \in T$, and $\alpha \in \Delta$ is a simple absolute root restricting to the simple relative root $\beta \in \Delta_{\rel}$, then
$$\cond_\beta({}^t \psi) = \cond_\beta(\psi) + \val(\alpha(t)).$$
Here we view $\alpha(t) \in (F^{\unr})^\times$, since $\alg{G}$ splits over $F^{\unr}$.
\end{prop}
\begin{proof}
Extending scalars to $F^{\unr}$, the relative root space decomposes as a direct sum of absolute root spaces 
$$\Lie{g}_\beta \otimes_F F^{\unr} \isom \bigoplus_{i = 1}^\ell \Lie{g}_{\alpha_i},$$
where $\{ \alpha_1, \ldots, \alpha_\ell \}$ is the set of absolute roots restricting to $\beta$, and $\Lie{g}_{\alpha_i}$ is a $F^{\unr}$-subspace of $\Lie{g} \otimes_F F^{\unr}$.

It follows that $\Ad(t)$ acts on $\Lie{g}_\beta \otimes_F F^{\unr}$ by scaling each summand by $\alpha_i(t)$.  Since $t \in T = \alg{T}(F^{\unr})^{\Fr}$, and the set $\{ \alpha_1, \ldots, \alpha_\ell \}$ forms a single $\Fr$-orbit, the integer $\val(\alpha_i(t))$ is independent of $i$.  The proposition follows, since the Lie algebra $\Lie{u}_\beta / \Lie{u}_{2 \beta}$ can be identified with $\Lie{g}_\beta$.
\end{proof}

\begin{cor}
Suppose that $t \in T$, and $\psi$ is a generic character of $U$.  Then $\cond_\beta({}^t \psi) = \cond_\beta(\psi)$ for all $\beta \in \Delta_{\rel}$ if and only if $t \in Z(G) \cdot T_\circ$.  
\end{cor}
\begin{proof}
By the previous proposition, $\cond_\beta({}^t \psi) = \cond_\beta(\psi)$ for all $\beta \in \Delta_{\rel}$ if and only if $\val(\alpha(t)) = 0$ for all $\alpha \in \Delta$.  This condition is equivalent to the condition that $t \in Z(G) \cdot T_\circ$.
\end{proof}

\begin{dfn}
Let $(\pi, V)$ be a genuine irrep of $\tilde G$, and $\psi$ a generic character of $U$.  A linear functional $\ell \From V \To \C$ is called a $\psi$-Whittaker functional if $\ell(\pi(u) v)= \psi(u) \cdot v$ for all $u\in U$ and $v\in V$. Write $\Wh_\psi(\pi)\defeq \Hom_U(\pi, \psi)$ for the space of $\psi$-Whittaker functionals for $\pi$. A genuine irrep $\pi$ of $\tilde G$ is called $\psi$-generic if $\dim \Wh_\psi(\pi)>0$.  We call $\pi$ generic if it is generic for some $\psi$. 
\end{dfn}

Write $Z = Z(G)$ and $\tilde Z$ for the preimage of $Z$ in $\tilde G$.  The group $\tilde Z$ is not necessarily abelian, and must not be confused with $Z(\tilde G)$.  Since unipotent groups split canonically in covers, we find that if $z \in \tilde z$ and $u \in U$, then $zuz^{-1} = u$.  Hence the space $\Wh_\psi(\pi)$ is naturally a genuine representation of $\tilde Z$.
%Genericity of characters and representations of finite reductive groups is defined similarly.

\subsection{Heredity}
\label{Rodier}
Rodier's heredity theorem (cf.~\cite{Rod}) for Whittaker functionals of linear algebraic groups holds for Brylinski-Deligne covering groups. For Kazhdan-Patterson covering groups, a proof with details is given by W. Banks \cite{Ban}.  Here we state the result for our class of covering groups.
\begin{thm}
Let $\alg{P} = \alg{M} \alg{N}$ be a standard $F$-parabolic subgroup of $\alg{G}$.  Let $\pi_M$ be an admissible genuine representation of $\tilde M$.  Inflate $\pi_M$ to $\tilde P$, and let $\pi_G$ be the representation $\Ind_{\tilde P}^{\tilde G} \pi_M$ obtained by normalized parabolic induction.

Let $\psi$ be a generic character of $U$.  Let $w_0$ be the longest element in the (relative) Weyl group, and let $\psi_M$ be the associated generic character of $U_M = w_0 U w_0^{-1} \cap M$ given by $\psi_M(w_0 u w_0^{-1}) = \psi(u)$.  Then we have
$$\dim \Wh_\psi( \pi_G) = \dim \Wh_{\psi_M}(\pi_M).$$
\end{thm}

We do not provide a proof here, since complete details are given in \cite{Ban}.  His proof adapts almost without change -- the essential ingredients are the Bruhat decomposition and the canonical splitting of covers over unipotent elements.  These ingredients carry over to our setting.

Since all depth zero representations arise from depth zero supercuspidal representations via parabolic induction, we limit our attention to supercuspidal representations in what follows.

%% subsection %%
\subsection{Depth zero supercuspidal representations}

We are interested in the $\epsilon$-genuine representations of $\tilde G$, where $\epsilon \From \mu_n \into \C^\times$ is a fixed injective character.  But since it is equivalent, and a bit more convenient, we consider representations of the full tame extension $\F_q^\times \into G' \onto G$ on which $\F_q^\times$ acts via the (usually not injective) character 
\begin{equation}
\label{eps'}
\epsilon' \From \F_q^\times \xrightarrow{(q-1)/n} \mu_n \xrightarrow{\epsilon} \C^\times.
\end{equation}
We call these {\em genuine} (or $\epsilon'$-genuine) representations of $G'$.

Let $x$ be a point in the building $\Build$.  From \eqref{respullback}, we note that the extension $\mu_n \into G_x' \onto G_x$ is canonically isomorphic to the pullback of an extension $\mu_n \into \bar M_x' \onto \bar M_x$ via $G_x \onto G_x / G_x^+ = \bar M_x$.  It follows that at any point $x \in \Build$, the extension $G'$ splits canonically over $G_x^+$.  

Hence it makes sense to talk about the $G_x^+$-fixed points in a representation of $G'$.  If $(\pi, V)$ is a genuine irrep of $G'$, we say that $(\pi, V)$ has {\em depth zero} if there exists a point $x \in \Build$ such that $V^{G_x^+} \neq 0$.  More generally, if $(\pi, V)$ is any smooth genuine representation of $G'$, we say that $(\pi, V)$ has {\em depth zero} if every irreducible subquotient of $(\pi, V)$ has depth zero.

For linear groups, depth zero representations were studied by L.~Morris \cite{Mor}, and by Moy and Prasad \cite{MP1} \cite{MP2}.  Some of their results were extended to covering groups by T.~Howard and the second author in \cite{HW}.  In the supercuspidal case, the depth zero representations are particularly easy to describe.

Consider a {\em vertex} $x \in \Build$; the group $\bar M_x' = \alg{\bar M}_x'(\F_q)$ is a finite group of Lie type, and fits into a central extension of such groups,
$$\F_q^\times \into \bar M_x' \onto \bar M_x.$$
A representation of $\bar M_x'$ will be called {\em genuine} if the central $\F_q^\times$ acts via the character $\epsilon'$ of \eqref{eps'}.  Any genuine representation of $\bar M_x'$ thus pulls back to a genuine representation of $G_x'$, trivial on $G_x^+$. 

By \cite[Proposition 3.9]{HW}, the following steps can be used to construct {\em every} genuine depth zero supercuspidal irrep of $G'$:
\begin{enumerate}
\item
Begin with an genuine cuspidal  (see \cite[\S 9.1]{Car}) irrep $\rho$ of the finite group $\bar M_x'$ at a {\em vertex} $x$.
\item
Pull back to a genuine representation $\rho$ of $G_x'$.
\item
Let $N_x'$ be the normalizer in $G'$ of $G_x'$.  Let $ \msc{R}(N_x', \rho)$ be the set of isomorphism classes of irreps of $N_x'$ which contain $\rho$ upon restriction to $G_x'$.  
\item
Choose any $\tau \in \msc{R}(N_x', \rho)$.  Then $\cind_{N_x'}^{G'} \tau$ is a genuine depth zero supercuspidal irrep of $G'$.
\end{enumerate}

\begin{dfn}
A {\em genuine cuspidal depth zero datum} for $G'$ is a triple $(x, \rho, \tau)$, where $x$ is a vertex in the building $\Build$, $\rho$ is a genuine  cuspidal irrep of $\bar M_x'$, and $\tau$ is a representation of $N_x'$ whose restriction to $G_x'$ contains (the pullback of) $\rho$.
\end{dfn}

For such a datum, let $\pi_{x, \tau} \defeq \cind_{N_x'}^{G'} \tau$, a genuine depth zero supercuspidal irrep of $G'$.  Given $\tau \From N_x' \To GL(V_\tau)$, the space of the representation $\pi_{x,\tau}$ is given by
\begin{align*}
\cind_{N_x'}^{G'} \tau \defeq \{ f \in C_c^\infty(G', V_\tau) : & \ f( n' g') = \tau(n') f(g') \\
& \text{ for all } n' \in N_x', \ g' \in G' \}.
\end{align*}

\subsection{A reduction of $\dim \Wh_\psi (\pi_{x,\tau})$}

Define $T_x' = T' \cap N_x'$, and similarly $T_x = T \cap N_x$.  Note that $N_x$ is the normalizer of the parahoric $G_x$, and $T \cap G_x = T_\circ$, the maximal compact subgroup of $T$.  As $N_x$ is the stabilizer of the image of $x$ in the reduced building of $G$, and $T$ acts on the reduced apartment through a faithful action of $T_{\ad, \circ} \backslash T_{\ad}$ (the torus in the adjoint quotient, modulo its maximal compact subgroup), we find that 
$$T_x \defeq T \cap N_x = T_\circ \cdot Z, \text{ and } T_x' = T_\circ \cdot Z'.$$
Here $Z = Z(G)$ denotes the center of $G$, and $Z'$ denotes its preimage in $G'$, which {\em may not be commutative}.  If $Y_{\ad}$ denotes the cocharacter lattice of $\alg{T}_{\ad}$, then we find an inclusion,
\begin{equation}
\label{TVal}
T_x' \backslash T' \equiv T_x \backslash T \into T_{\ad, \circ} \backslash T_{ad} \xrightarrow{\val} Y_{\ad}^{\Fr}.
\end{equation}
If $t' \in T'$, we write $t$ for its image in $T$, and we write $\val(t') \in Y_{\ad}^{\Fr}$ for its valuation described above.

Given $(\tau,V_\tau) \in \msc{R}(N_x', \rho)$ as before, and $f \in \cind_{N_x'}^{G'} \tau$, define a function $\eta_f \From G' \To V_\tau$ by
$$\eta_f(g') = \int_U f( g' \cdot u) \psi(u)^{-1} du.$$
Restricting to $T'$, the function $\eta_f \From T' \To V_\tau$ satisfies two identities:
\begin{enumerate}
\item[(E1)]
$\eta_f( s' \cdot t') = \tau(s') \eta_f(t')$ for all $s' \in T_x'$ and $t' \in T'$.
\item[(E2)]
$\tau(u) \eta_f( t') = \psi( t^{-1} u t) \cdot \eta_f(t')$ for all $u \in U \cap G_x$ and $t' \in T'$.
\end{enumerate}

Motivated by this, define $E(\tau, \psi)$ to be the space of all smooth, compactly supported functions $\eta \From T' \To V_\tau$ satisfying (E1) and (E2) above.  Then the map $f \mapsto \eta_f$ defines a linear map,
$$\pi_{x, \tau} = \cind_{N_x'}^{G'} \tau \To E(\tau, \psi).$$
By the well-known characterization of the twisted Jacquet module, this factors through an injective linear map,
$$I \From \left( \pi_{x,\tau} \right)_{U, \psi} \into E(\tau, \psi).$$
In fact, adapting a proposition of Blondel (\cite[Proposition 2]{Blo}), we have
\begin{prop}
This is an isomorphism, $I \From \left( \pi_{x,\tau} \right)_{U, \psi} \xrightarrow{\sim} E(\tau, \psi)$.  
\end{prop}
\begin{proof}
Our proof here is the same as \cite{Blo}, included for completeness.  It remains to show that $I$ is surjective.  For this, we decompose $E(\tau, \psi)$ according to the cosets of $T_x'$ in $T'$, using \eqref{TVal}.  If $y \in Y_{\ad}^{\Fr}$, define
$$E_y(\tau, \psi) = \{ \eta \in E(\tau, \psi) : \eta(t') \neq 0 \text{ implies } \val(t') = y \}.$$

We have a direct sum decomposition,
$$E(\tau, \psi) = \bigoplus_{y \in Y_{\ad}^{\Fr}} E_{y}(\tau, \psi).$$
Suppose $0 \neq \eta \in E_{y}(\tau, \psi)$, choose $a' \in T'$ such that $\val(a') = y$, and define a function $f \From G' \To V_\tau$ by
$$f(n' \cdot a') = \tau(n') \cdot \eta(a') \text{ for all } n' \in N_x'.$$
We put $f(g') = 0$ if $g' \not \in N_x' \cdot a'$.  By construction, $f \in \cind_{N_x'}^{G'} \tau$, and the computation below demonstrates that $\eta_f$ is a nonzero multiple of $\eta$.
\begin{align*}
\eta_f(t') &= \int_U f( t' \cdot u) \psi(u)^{-1} du, \\
&= \int_U f( t u t^{-1} \cdot t') \psi(u)^{-1} du, \\
&= \int_U \delta(t)^{-1} f(u \cdot t') \psi(t^{-1} u t)^{-1} du.
\end{align*}
In the last line, $\delta(t)$ denotes the modular character for the action of $t$ on $U$.  We have $f(u \cdot t') = 0$, unless $u \in U \cap N_x' = U \cap G_x'$ and $t' = s' a'$ for some $s' \in T_x'$.  It follows that
\begin{align*}
\eta_f(t') &= \int_U \delta(t) f(u \cdot t') \psi(t^{-1} u t)^{-1} du, \\
&= \int_{U \cap G_x'} \delta(s a) \psi(t^{-1} u t)^{-1} \cdot f( us' \cdot a') du, \\
&= \int_{U \cap G_x'} \delta(a) \psi(t^{-1} u t)^{-1} \cdot \tau(us') \eta(a')  du, \\
&= \delta(a) \int_{U \cap G_x'} \psi(t^{-1} u t)^{-1} \cdot \tau(u) \eta(t') du  \quad \text{ by (E1)}, \\
&= \delta(a) \int_{U \cap G_x'} \psi(t^{-1} u t)^{-1} \psi(t^{-1} u t) \cdot \eta(t') \quad \text{ by (E2)}, \\
&= \delta(a) \int_{U \cap G_x'} \eta(t') = \delta(a) \cdot \Vol(U \cap G_x') \cdot \eta(t').
\end{align*}
Since $\eta_f$ is a nonzero multiple of $\eta$, the map $I$ is surjective.
\end{proof}

If $t' \in T'$ and $\val(t') = y$, then 
$$\dim \Hom_{U \cap G_x'} \left( {}^{t} \psi, \tau \right) = \dim E_y(\tau, \psi),$$
by condition (E2).   In particular, the multiplicity $\dim \Hom_{U \cap G_x'} \left( {}^{t} \psi, \tau \right)$ depends only upon $y$, and we write $m_y(\psi, \tau)$ for this multiplicity.

\begin{cor}
\label{SuppPsi}
For every genuine depth zero supercuspidal irrep $\pi_{x,\tau}$,
$$\dim \Wh_\psi( \pi_{x, \tau}) = \sum_{y \in \val(T') \subset Y_{\ad}^{\Fr}} m_y(\psi, \tau).$$
\end{cor}

\subsection{Uniqueness of support}

The formula of Corollary \ref{SuppPsi} relates the space of Whittaker models for $\pi_{x,\tau}$ to a family of spaces of $(U_x, {}^t \psi)$-invariant functionals for the representation $\tau$ of $N_x'$, as $t$ varies.  Here we write $U_x = U \cap G_x$; since covers split over unipotent subgroups, we identify $U_x$ with $U \cap G_x'$.  As $y = \val(t')$ varies, the character ${}^t \psi$ can become ``too trivial'' or ``too nontrivial'' for $\tau$ to support such a functional.  Here we prove that there is at most one value of $y$ for which $\tau$ has such an invariant functional.

For this, we review the structure of root subgroups and filtrations of $U_x$.  Recall that $G_x$ arises as the group of $\OO$-points of the Bruhat-Tits group scheme $\sch{G}_x$ over $\OO$, with connected special fibre $\alg{\bar M}_x$.  The subgroup $G_x^+ \subset G_x$ is identified with the kernel of reduction,
$$G_x^+ = \Ker \left( \sch{G}_x(\OO) \To \sch{G}_x(\F_q) = \alg{\bar M}_x(\F_q) \right).$$
Recall that $\Phi_x$ is the set of roots $\alpha$ for which $\alpha(x) \in \Z$; these are identified with the absolute roots of $\alg{\bar M}_x$ with respect to $\alg{\bar T}$.  As $\alg{\bar M}_x$ is also quasisplit, the {\em relative} roots of $\alg{\bar M}_x$ with respect to $\alg{\bar S}$ correspond to $\Fr$-orbits on $\Phi_x$.  In this way, we identify the relative roots of $\alg{\bar M}_x$ with a subset $\Phi_{x, \rel} \subset \Phi_{\rel}$.  

To describe the relative {\em simple} roots of $\alg{\bar M}_x$, we proceed as follows:  the Borel subgroup $\sch{B}$ over $\O$ determines an Iwahori subgroup $J \subset G$, consisting of those elements of $G_\circ$ whose reduction lies in $\bar B$.  This corresponds to a chamber in the building of $G$.  Since $G$ acts transitively on the chambers of $\Build$, one can assume that $x$ lies in the closure of the chamber corresponding to $J$.  The chamber determines a set of simple {\em affine} roots $\hat \Delta_{\rel}$.  The relative simple roots of $\alg{\bar M}_x$ can be identified with a subset $\Delta_{x, \rel}$ of the gradients of the {\em affine} simple roots $\hat \Delta_{\rel}$ -- those whose affine root hyperplanes vanish at $x$.  Details can be found in the exposition of Tits \cite[\S 1.9]{Tits}. 

Recall that the hyperspecial point $\circ$ determined filtrations on root subgroups.  Namely, if $\beta \in \Phi_{\rel}$, we found a valuation
$$\val \From U_\beta / U_{2 \beta} \To \Z \cup \{ \infty \},$$
and $\val(u) \geq 0$ if and only if $u$ has a representative in $U_\beta \cap G_\circ$.  

If $\beta \in \Phi_{x, \rel}$, then an element $u \in U_\beta / U_{2 \beta}$ can be represented by an element of $U_\beta \cap G_x$ if and only if $\val(u) \geq \beta(x)$.  The subgroup $G_x^+$ is one step deeper:  $u$ can be represented by an element of $U_\beta \cap G_x^+$ if and only if $\val(u) \geq \beta(x) + 1$.

The following lemma is very close to \cite[Lemma 6.1.2]{DR} and \cite[Lemma 3.1]{DR2}.  The proof is fundamentally the same as theirs, though our notation and goals differ slightly.  
\begin{lm}
Let $\rho$ be a cuspidal irrep of $\bar M_x'$, pulled back to $G_x'$, and suppose $\tau \in \msc{R}(N_x', \rho)$.  Let $\psi \From U \To \C^\times$ be a nontrivial character.  If $\Hom_{U_x}(\psi, \tau) \neq 0$ then $x$ is a hyperspecial point and 
$$\cond_\beta(\psi) = \beta(x) + 1 \text{ for all } \beta \in \Delta_{x, \rel} = \Delta_{\rel}.$$
\end{lm}
\begin{proof}
Suppose that $\Hom_{U_x}(\psi, \tau) \neq 0$ in what follows.

First, we exploit the cuspidality of $\rho$.  Let $\alg{\bar P}_\beta = \alg{\bar L}_\beta \alg{\bar R}_\beta$ be the maximal parabolic subgroup corresponding to the subset $\Delta_{x, \rel} - \{ \beta \}$ of simple relative roots.  Given $\psi \From U \To \C^\times$, we find that $\psi$ is trivial on $[U,U]$.  If $\psi$ were also trivial on $U_\beta \cap G_x$, then $\psi$ would restrict to a trivial character on the unipotent radical $\bar R_\beta$.  Since $\rho$ is cuspidal, and thus $\tau$ restricted back to $\bar G_x'$ is cuspidal, we find that $\tau$ has no nontrivial coinvariants for $\bar R_\beta$.  Therefore, $\psi$ is nontrivial on $U_\beta \cap G_x$.

% Roots in [U,U] are sums of at least two simple roots (possibly a root with itself).
% Roots of R_\beta are those which are \beta, and \beta + something.
% So if a character is trivial on [U,U] and trivial on the \beta space, it is trivial on R_\beta.

We claim that $\Delta_{\rel} = \Delta_{x, \rel}$.  For they have the same cardinality (since $x$ is a vertex) and can both be viewed as subsets of $\hat \Delta_{\rel}$.  If $\Delta_{\rel} \neq \Delta_{x, \rel}$, there would be a $\beta \in \Delta_{x, \rel}$ contained in $\Phi^+$ but not in $\Delta$ (a gradient of a simple affine root which is not a simple root).  But $\psi$ is trivial on $[U,U]$ and hence on $U_\beta$ for such a $\beta$, a contradiction.    

Thus $\Delta_{x, \rel} = \Delta$, and for every simple root $\beta \in \Delta_{\rel}$, there exists an integer $k_\beta$ such that $\beta(x) = k_\beta$.  There exists an element $t \in T_{\ad}$ such that $\val \alpha(t) = k_\beta$ for every absolute root $\alpha$ restricting to $\beta$ (the coroot lattice equals the coweight lattice in the adjoint group $\alg{G}_{\ad}$).  It follows that $t$ translates $\circ$ to $x$ -- the elements of $T_{\ad}$, as $F$-rational automorphisms of $G$, act on the building of $G$, and therefore $x$ is hyperspecial.  More detail can be found in the proof of \cite[Lemma 6.1.2]{DR}, upon which this argument is based.

We have proven that $x$ is hyperspecial.  Moreover, the nontriviality of $\psi$ on $U_\beta \cap G_x$ now implies that
$$\cond_\beta(\psi) \geq \beta(x) + 1 \text{ for all } \beta \in \Delta_{x, \rel} = \Delta_{\rel}.$$

To bound the conductor in the other direction, we note that $\rho$ arises by pulling back a representation of $\bar M_x'$, via $G_x' \To G_x' / G_x^+ = \bar M_x'$.  We find that $\rho$ is trivial on $G_x^+$.  Since $N_x'$ stabilizes the subgroup $G_x^+$, we find that $\tau$ also factors through $N_x' / G_x^+$.  Hence $\Hom_{U_x}(\psi, \tau) = 0$ unless $\psi$ is trivial on $U \cap G_x^+$.  This condition implies that if $\beta \in \Delta_{x, \rel} = \Delta_{\rel}$, then $\psi$ is trivial on $U_\beta \cap G_x^+$, and thus $\psi_\beta$ is trivial on all elements $u \in U_\beta / U_{2 \beta}$ for which $\val(u) \geq \beta(x) + 1$.  We find that
$$\cond_\beta(\psi) \leq \beta(x) + 1 \text{ for all } \beta \in \Delta_{x, \rel}.$$
\end{proof}

\begin{thm}
\label{OneSupp}
Suppose that $\pi_{x,\tau}$ is a genuine depth zero supercuspidal irrep of $G'$.  If $\pi_{x, \tau}$ is generic, then $x$ is hyperspecial, and there exists a unique $y \in Y_{\ad}^{\Fr}$ such that
$$\dim \Wh_\psi( \pi_{x, \tau} ) = m_y(\psi, \tau).$$
\end{thm}
\begin{proof}
Corollary \ref{SuppPsi} implies that if $\pi_{x, \tau}$ is generic, then $m_y(\psi, \tau)$ is nonzero for some $y$; thus $\Hom_{U_x}({}^t \psi, \tau) \neq 0$ for some $t' \in T'$.  Therefore $x$ is hyperspecial by the previous lemma.

To complete the proof, it suffices to show that $m_y(\psi, \tau) \neq 0$ for at most one value of $y$.  To see this, we have 
$$\cond_\beta( {}^t \psi) = \cond_\beta(\psi) + \val(\alpha(t))$$
for every $t' \in T'$, every relative root $\beta$, and every absolute root $\alpha$ restricting to $\beta$.  If $0 \neq \val(t') \in Y_{\ad}^{\Fr}$, there exists an absolute root $\alpha$ such that $\val(\alpha(t')) \neq 0$, and hence
$$\cond_\beta( {}^t \psi) \neq \cond_\beta(\psi).$$

In other words, as $\val(t')$ varies in $Y_{\ad}^{\Fr}$, the conductors $\cond_\beta({}^t \psi)$ vary.  But the previous lemma demonstrates that $m_y(\psi, \tau)$ can only be nonzero for one set of values of $\{ \cond_\beta({}^t \psi) : \beta \in \Delta_{\rel} \}$.  Hence $m_y(\psi, \tau) \neq 0$ for at most one value of $y$.
\end{proof}

\subsection{Clifford theory}

Let $(x, \rho, \tau)$ a genuine cuspidal depth zero datum.  Let $N_{x, \rho}'$ be the stabilizer of $\rho$ in $N_x'$, i.e.,  the set of all $n \in N_x'$ such that ${}^{(n')}(\rho)$ is isomorphic to $\rho$ as representations of $G_x'$.  

If $x$ is hyperspecial (as will be the case when $\pi_{x,\tau}$ is generic), then $N_x = Z \cdot G_x$ (we learned this from \cite[\S 6.1]{DR}).  But moreover, $G_x$ contains the maximal compact subgroup of $Z$ since $G_x$ is a maximal compact subgroup of $G$.  If $Z_\circ$ denotes the maximal compact subgroup of $Z$, then we can write $Z = A \cdot Z_\circ$, where $A = \alg{A}(F)$ and $\alg{A}$ is the maximal split torus contained in $\alg{Z}(\alg{G})$.  Thus we have
$$N_x = A \cdot G_x, \quad N_x' = A' \cdot G_x',$$
where $A'$ denotes the preimage of $A$ in $G'$ as usual.  One can go even a bit further -- the cocharacter lattice of $\alg{A}$ is naturally identified with $Y^{W \rtimes \Fr}$, and $A / A_\circ$ can thus be identified with this lattice.  Here $W \rtimes \Fr$ denotes the group of automorphisms of $Y$ generated by the absolute Weyl group $W$ and the Frobenius automorphism.  Concretely, if $\varpi$ is a uniformizer, then every element of $N_x$ can be written as $y(\varpi) \cdot k$ for some $y \in Y^{W \rtimes \Fr}$ and $k \in G_x$.  This gives an isomorphism,
\begin{equation}
\label{NmodP}
\frac{N_x'}{G_x'} \isom \frac{N_x}{G_x} \isom Y^{W \rtimes \Fr}
\end{equation}

The following is a consequence of Clifford theory.
\begin{prop}
Assume that $x$ is hyperspecial.  Then there exists an extension of $\rho$ to an irrep $\rho_e$ of $N_{x, \rho}'$, and there exists an isomorphism of representations of $N_x'$,
$$\tau \isom \Ind_{N_{x,\rho}'}^{N_x'} \rho_e.$$
\end{prop}
\begin{proof}
The only subtlety is the extension of $\rho$ to its stabilizing subgroup $N_{x,\rho}'$.  For this, define $Y_{x, \rho} = N_{x,\rho}' / G_x'$, and note that
$$Y_{x,\rho} = \frac{N_{x, \rho}'}{ G_x'} \subset \frac{N_x'}{G_x'} \isom Y^{W \rtimes \Fr}.$$
In particular, $Y_{x,\rho}$ is a finite-rank free abelian group.  The obstruction to extending $\rho$ to a representation of $N_{x,\rho}'$ is a cohomology class in $H^2(Y_{x,\rho}, \C^\times)$, which vanishes since $Y_{x,\rho}$ is free.  Thus $\rho$ extends, and the rest is standard Clifford theory.
\end{proof}
The following diagram summarizes the groups and representations.
$$\begin{tikzcd}
\bar M' & G_x' \onarrow{l}{\text{ mod } G_x^+} \inarrow{r}{\subset} & N_{x,\rho}' \inarrow{r}{\subset} & N_x'  \inarrow{rrr}{\subset} & & & G' \\
\rho \arrow[mapsto]{r}{\text{pull back}} & \rho \arrow[mapsto]{r}{\text{extend}} & \rho_e \arrow[mapsto]{r}{\text{induce}} & \tau \arrow[mapsto]{rrr}{\text{compactly induce}} & & &  \pi_{x,\tau}
\end{tikzcd}$$

\subsection{Bounding the dimension}

For linear groups, the method of Gelfand-Kazhdan \cite{GK} yields uniqueness of Whittaker models for all quasisplit groups -- see \cite{Rod2}, \cite{Shal}.  For covering groups, the method of Gelfand-Kazhdan seems difficult to adapt (at least, when $\tilde T$ is nonabelian).  But for depth zero representations, one can start from uniqueness of Whittaker models for finite reductive groups, and deduce a multiplicity bound for Whittaker models for $\tilde G$.

If $x$ is a hyperspecial point, write $\alg{\bar U}_x$ for the unipotent radical of the Borel subgroup $\alg{\bar B}_x \subset \alg{\bar M}_x$, associated to our previous choice of chamber.  If $\bar \psi \From \bar U_x \To \C^\times$ is a character, then as we found for groups over $F$, we may decompose $\bar \psi$ as a product of characters,
$$\bar \psi = \prod_{\beta \in \Delta_{x, \rel}} \bar \psi_\beta, \quad \bar \psi_\beta \From \bar U_{x, \beta} / \bar U_{x, 2 \beta} \To \C^\times.$$
We say that $\bar \psi$ is {\em generic} if $\bar \psi_\beta$ is nontrivial for every $\beta \in \Delta_{x, \rel}$.  

If $\rho$ is an irreducible representation of $\bar M_x'$, and $\bar \psi$ is a generic character, then we write $\Wh_{\bar \psi}(\rho)$ for the space $\Hom_{\bar U_x}(\rho, \bar \psi)$ of Whittaker models.  By Frobenius reciprocity, its dimension equals the multiplicity of $\rho$ in the Gelfand-Graev representation $\Ind_{\bar U_x}^{\bar M_x'} \psi$.  But this representation is multiplicity-free, following Gelfand-Graev \cite{GelGr} (in special cases), Yokonuma \cite{Yok}, and Steinberg \cite[Theorem 49]{Ste}.
\begin{thm}
The representation $\Ind_{\bar U_x}^{\bar M_x'} \bar \psi$ is multiplicity-free, and hence $\dim \Wh_{\bar \psi}(\rho) \leq 1$ for every irrep $\rho$ of $\bar M_x'$ and every generic character $\bar \psi$ of $\bar U_x$.
\end{thm}

From this, we find a formula for the dimension of the space of Whittaker models for genuine depth zero supercuspidals.
\begin{thm}
\label{Bound1}
Let $\pi_{x, \tau}$ be a genuine, depth zero supercuspidal irrep of $G'$.  If $\pi_{x, \tau}$ is $\psi$-generic, then $x$ is hyperspecial and
$$\dim \Wh_\psi( \pi_{x, \tau} ) = [ N_x' : N_{x, \rho}'] = [Y^{W \rtimes \Fr} : Y_{x,\rho}].$$
\end{thm}
\begin{proof}
By Theorem \ref{OneSupp}, we know that $x$ is hyperspecial and
$$\dim \Wh_\psi( \pi_{x, \tau} ) = m_y( \psi, \tau),$$
for some $y = \val(t') \in Y_{\ad}^{\Fr}$.  Define $\bar \psi$ to be the character of $\bar U_x$ obtained by restricting ${}^{t} \psi$ to $U_x$; as $x$ is hyperspecial, $\Delta_{\rel} = \Delta_{\rel, x}$ and we find that $\bar \psi$ is generic.

Recall furthermore that $\tau$ is an irrep of $N_x'$, and 
$$\tau \isom \Ind_{N_{x, \rho}'}^{N_x'} \rho_e, \text{ and } \Res_{G_x'}^{N_{x, \rho}'} \rho_e = \rho.$$
Thus we have
\begin{equation}
\label{ResInd}
\Res_{G_x'}^{N_x'} \tau \isom \bigoplus_{i=1}^\ell \rho_i,
\end{equation}
where $\ell = [ N_x' : N_{x, \rho}'] = [Y^{W \rtimes \Fr} : Y_{x,\rho}]$, and $\rho_i$ is an irrep of $G_x'$ trivial on $G_x^+$.  

Since $x$ is hyperspecial, the normalizer $N_x = Z \cdot G_x$.  As the representations $\rho_i$ form a single $N_x'$-orbit, and $G_x \subset N_x$, there exist $z_1', \ldots, z_\ell' \in Z'$ such that
$$\rho_i \isom {}^{z_i'} \rho_e \text{ as representations of } N_{x, \rho}'.$$
Note that although $Z'$ (the preimage of $Z$ in $G'$) is not necessarily central in $G$, elements $z' \in Z'$ act trivially by conjugation on $U$.  This follows from the canonical -- and thus conjugation-invariant -- splitting of $U$ in $G'$.  Therefore,
$$\dim \Wh_{\bar \psi}(\rho_i) = \dim \Wh_{\bar \psi}(\rho), \text{ for all } 1 \leq i \leq \ell.$$

From \eqref{ResInd}, it follows that
$$\dim \Wh_\psi( \pi_{x, \tau} ) = m_y( \psi, \tau) = \sum_{i=1}^\ell \dim \Wh_{\bar \psi}(\rho_i) = \ell.$$
The last inequality follows from the uniqueness of Whittaker models for the finite reductive group $\bar M_x'$.
\end{proof}

\subsection{A refinement}

To refine Theorem \ref{Bound1}, we analyze the index
$$[ Y^{W \rtimes \Fr} : Y_{x, \rho} ],$$
noting that 
$$Y_{x,\rho} = \{ y \in Y^{W \rtimes \Fr} : \rho \isom {}^{y(\varpi)} \rho \}.$$

Equivalently, we analyze the conjugation action of $A / A_\circ$ on the isomorphism classes of representations of $G_x'$.  For this, note that if $a \in A$ and $g \in G$, then $[a,g] = 1$ since $A$ is central in $G$.  Therefore, choosing lifts $a' \in A'$ and $k' \in G_x'$, we have
$$[a',k'] \defeq a' \cdot k' \cdot (a')^{-1} \cdot (k')^{-1} \in \F_q^\times.$$
\begin{lm}
\label{ChiLem}
If we define $\chi_a \From G_x \To \F_q^\times$ by $\chi_a(k) = [a', k']$, then $\chi_a$ is a homomorphism, and the map $a \mapsto \chi_a$ defines a homomorphism,
$$\chi \From A \To \Hom(G_x, \F_q^\times).$$
\end{lm}
\begin{proof}
In a general group, we have the commutator identities,
$$[x,y] = [y,x]^{-1} \text{ and } [xy, z] = [ x, [y,z]] [y,z] [x,z].$$
In our setting, we have $[a', k'] \in \F_q^\times$, which is central.  It follows that
$$[a_1' a_2', k'] = [a_2', k'] [a_1', k'] \text{ for all } a_1, a_2 \in A, k \in G_x.$$
$$[a', g_1' g_2'] = [a', g_1'] [a', g_2'].$$
The result follows.
\end{proof}

The homomorphisms $\Hom(\alg{\bar M}_x, \alg{\bar G}_m)$ (in the category of algebraic groups over $\F_q$), can be identified with $\Hom_{\Z[\Fr]}(Y / Y_{\SC}, \Z)$, since $x$ is hyperspecial.  The identification can be made as follows:  if $\alg{\bar \chi} \From \alg{\bar M}_x \To \alg{\bar G}_m$ is a homomorphism, and $y \in Y$, then there is a unique integer $\kappa(y)$ satisfying
$$\alg{\bar \chi}(y(z)) = z^{\kappa(y)} \text{ for all } z \in \bar \F_q^\times.$$
This $\kappa$ factors through a homomorphism $\kappa \From Y / Y_{\SC} \To \Z$, and this is compatible with $\Fr$-action throughout.  The map $\alg{\bar \chi} \mapsto \kappa$ gives the identification.

For $\alg{\bar \chi} \in \Hom(\alg{\bar M}_x, \alg{\bar G}_m)$, define $\bar \chi \From \bar M_x \To \F_q^\times$ to be the resulting homomorphism on points.  Recall that $\epsilon' \From \F_q^\times \To \C^\times$ is a fixed homomorphism, factoring through $\mu_n$.  Thus $\epsilon' \circ \bar \chi \From \bar M_x \To \C^\times$ is a character.  Pulling back, we view $\epsilon' \circ \bar \chi$ as a character of $G_x$.

\begin{prop}
\label{ChiProp}
For all $y \in Y^{W \rtimes \Fr}$, there is an isomorphism of representations of $G_x'$,
$${}^{y(\varpi)} \rho \isom \rho \otimes (\epsilon' \circ \bar \chi),$$
where $\alg{\bar \chi} \in \Hom(\alg{\bar M}_x, \alg{\bar G}_m)$ corresponds to $\kappa_y \in \Hom_{\Z[\Fr]}(Y / Y_{\SC}, \Z)$, with 
$$\kappa_y(y') = B_Q(y, y').$$
\end{prop}
\begin{proof}
The previous lemma, and the fact that $\rho$ is $\epsilon'$-genuine, guarantee that
$${}^{y(\varpi)} \rho \isom \rho \otimes (\epsilon' \circ \bar \chi_a),$$
where $\chi_a \From G_x \To \F_q^\times$ is the character arising from $a = y(\varpi)$.  So it remains to prove that $\chi_a$ arises from the character $\alg{\bar \chi}$ described in this proposition.

The character $\chi_a$ is trivial on the unipotent subgroups of $G_x$, since splittings of $G'$ over unipotent subgroups are canonical, and hence conjugation-invariant.  Thus the character $\chi_a$ is determined by its values on the maximal torus $\bar T \subset G_x$.  It suffices to work over $F^{\unr}$ and $\bar \F_q$, since everything in sight is $\Fr$-equivariant.  (See \cite[\S 12.9-12.12]{BD}, for example.)  If $y'(\bar z) \in \bar T$ for some $y' \in Y$ and $\bar z \in \bar  \F_q^\times$, then for any lift $z \in \OO^{\unr}$ of $\bar z$,  the commutator $\chi_a(y'(z))$ in $G'$ is given by the formula of \cite[Corollary 3.14]{BD},
\begin{equation}
\label{ChiComm}
\chi_a(y'(z)) = \partial \{ \varpi, z \}^{B_Q(y, y')} = \bar z^{B_Q(y,y')}.
\end{equation}
This agrees with the claim of the proposition.
\end{proof}

\begin{cor} \label{cor-key}
For $y \in Y^{W \rtimes \Fr}$, write $\bar \chi_y \From \bar M_x \To \F_q^\times$ for the character corresponding to $\kappa_y$.  Then 
$$Y_{x, \rho} = \{ y \in Y^{W \rtimes \Fr} : \rho \isom \rho \otimes (\epsilon' \circ \bar \chi_y) \}.$$
Moreover, $Y_{x, \rho} \supset Y^{W \rtimes \Fr} \cap Y_{Q,n}$.
\end{cor}
\begin{proof}
The main statement follows directly from the previous proposition.  The fact that $Y_{x, \rho} \supset Y^{W \rtimes \Fr} \cap Y_{Q,n}$ follows from the formula in \eqref{ChiComm}.  Indeed, if $y \in Y_{Q,n}$, we find that $B_Q(y,y') \in n \Z$, and so 
$$\epsilon'(\bar x^{B_Q(y, y')} ) = \epsilon \left( \bar x^{ (q-1) B_Q(y,y') / n} \right) = 1.$$
\end{proof}

As we shall see, this gives a practical method of describing $Y_{x, \rho}$ in some important examples.  From this we can compute the index of $[Y^{W \rtimes \Fr} : Y_{x, \rho}]$, which equals the dimension of the space of Whittaker functionals.  To summarize, Theorem \ref{Bound1} and Corollary \ref{cor-key} gives us
\begin{equation}
\label{WhBound}
\dim \Wh_\psi(\pi_{x, \tau}) =  [Y^{W \rtimes \Fr} : Y_{x, \rho}] \text{ divides }  [ Y^{W \rtimes \Fr} : Y^{W \rtimes \Fr}  \cap Y_{Q,n} ] .
\end{equation}
whenever $\pi_{x,\tau}$ is a $\psi$-generic representation.

%%%
\subsection{A few easy cases}

In what follows, we maintain the notation from before:  $\pi = \pi_{x,  \tau}$ will be a $\psi$-generic depth zero supercuspidal irrep of $\tilde G$.  We will describe the dimension of the space of Whittaker functionals, $\dim \Wh_\psi(\pi)$, in a few easy cases.  For reference, we recall
\begin{equation}
\label{Wh1}
\dim \Wh_\psi(\pi) =  [Y^{W \ltimes \Fr} : Y_{x,\rho} ] = [N_x' : N_{x,\rho}'], \quad \text{(Theorem \ref{Bound1})}
\end{equation}

\begin{equation}
\label{Wh2}
\dim \Wh_\psi(\pi) = [Y^{W \ltimes \Fr} : Y_{x,\rho} ] \text{ divides } [Y^{W \ltimes \Fr} : Y^{W \ltimes \Fr} \cap Y_{Q,n} ].
\end{equation}

\begin{cor}
If $\alg{G}$ is semisimple then $\dim \Wh_\psi(\pi) = 1$.
\end{cor}
\begin{proof}
In this case $Z \subset G_x$ and so $N_x' = G_x'$ (the point $x$ is hyperspecial) and $[N_x' : N_{x,\rho}'] = 1$.  Theorem \ref{Bound1} gives the result.
\end{proof}

The opposite case occurs when $\alg{G} = \alg{T}$, an unramified torus.  Every genuine irrep $\pi$ of $\tilde T$ is supercuspidal and generic.  The space of Whittaker functionals $\Wh_\psi(\pi)$ is the entire space $\pi$, since $\alg{U}$ is trivial.  In this case, we find
\begin{prop}
If $\alg{G} = \alg{T}$, then $\dim \Wh_\psi(\pi) = [Y^\Fr : Y_{Q,n}^\Fr]$.
\end{prop}
\begin{proof}
The dimension of $\pi$ is computed in \cite[\S 6]{We1}.
\end{proof}
Note that in the case $\alg{G} = \alg{T}$, $W$ is trivial and the dimension of $\Wh_\psi(\pi)$ agrees with the upper bound of \eqref{Wh2}.

Between these cases, we have an easy bound.
\begin{prop}
Let $\alg{Z}^\circ$ be the connected component of the center of $\alg{G}$.  Thus $\alg{Z}^\circ$ is an unramified torus, with cocharacter lattice $Y^W$.  Then we have,
$$\Zind(\tilde Z^\circ) = [Y^{W \rtimes \Fr} : (Y^W)_{Q,n}^{\Fr}] \text{ divides } \dim \Wh_\psi(\pi).$$
Here, the lattice $(Y^W)_{Q,n}^{\Fr}$ is defined by
$$(Y^W)_{Q,n}^{\Fr} = \{ y \in Y^{W \rtimes \Fr} : B_Q(y, y') \in n \Z \text{ for all } y' \in Y^W \}.$$
\end{prop}
\begin{proof}
The space $\Wh_\psi(\pi)$ is naturally a genuine representation of $\tilde Z^\circ$.  Since $\Wh_\psi(\pi)$ is finite-dimensional, and all genuine representations of $\tilde Z^\circ$ have the same dimension $\Zind(\tilde Z^\circ)$ (computed in \cite[\S 6]{We1}), the result follows.
\end{proof}

To summarize, we have inclusions of lattices,
\begin{equation}
\label{SqueezeWh}
(Y^{W \rtimes \Fr} \cap Y_{Q,n}) \  \subset  \ Y_{x,\rho}  \ \subset (Y^W)_{Q,n}^\Fr  \ \subset \ Y^{W \rtimes \Fr},
\end{equation}
giving upper and lower bounds on $\dim \Wh_\psi(\pi) = [Y^{W \rtimes \Fr} : Y_{x,\rho}]$.

%%%

\section{Deligne-Lusztig representations}
We maintain the notation of $(x, \rho, \tau)$, with $\pi_{x, \tau}$ being $\psi$-generic.  Thus $x$ is hyperspecial, and $\alg{\bar G}_x = \alg{\bar M}_x$ is a reductive group over $\F_q$.  The residual extension is $\alg{\bar G}_m \into \alg{\bar G}_x' \onto \alg{\bar G}_x$.  The representation $\rho$ is a cuspidal irrep of $\bar G_x'$; such representations have been studied in detail by Deligne and Lusztig in \cite{DL}.  Such irreps of $\bar G_x'$ arise from characters of tori, as we review below.

\subsection{Construction of Deligne-Lusztig representations}

An exercise in Galois cohomology demonstrates that the $\bar G_x'$-conjugacy classes of maximal $\F_q$-tori in $\alg{\bar G}_x'$ are parameterized by $\Fr$-conjugacy classes in the Weyl group $W$.  Since $x$ is hyperspecial, we are identifying the Weyl group of $\alg{\bar G}_x'$ with respect to $\alg{\bar T}'$ and the Weyl group of $\alg{G}$ with respect to $\alg{T}$.  See \cite[Lemma 4.2.1]{DeB} for a proof of this parameterization (cf.~\cite[\S 1.8]{DL}).  

For $w \in W$, we write ${}^w \alg{\bar T}'$ for a corresponding maximal $\F_q$-torus in $\alg{\bar G}_x'$; its image in $\alg{\bar G}_x$ will be denoted ${}^w \alg{\bar T}$.  Note that ${}^w \alg{\bar T}'$ contains the central $\alg{\bar G}_m$ of the residual extension.  The $\F_q$-points of ${}^w \alg{\bar T}'$ can be described as
$${}^w \bar T' \defeq {}^w \alg{\bar T}'(\F_q) = \{ t' \in \alg{\bar T}'(\bar \F_q) : w(\Fr(t')) = t' \}.$$
As sets, we identify the character and cocharacter lattices of ${}^w \alg{\bar T}$ with those of $\alg{\bar T}$; the Frobenius action is twisted however by $w$, so we write $\Fr_w(\xi) = w(\Fr(\xi))$ for all $\xi \in X$ and $\Fr_w(y) = w(\Fr(y))$ for all $y \in Y$.  The normalization is chosen here so that for all $x \in X$, $y \in Y$, we have
$$\langle \Fr_w(x), \Fr_w(y) \rangle = \langle x, y \rangle.$$
The Weyl group of $\alg{\bar G}_x$ with respect to ${}^w \alg{\bar T}$ is similarly identified with $W$ as sets, but the Frobenius action is twisted to become $\Fr_w(w') = w \Fr(w') w^{-1}$ for all $w' \in W$.

We say that ${}^w \alg{\bar T}$ is {\em minisotropic} if it does not lie in any proper $\F_q$-parabolic subgroup of $\alg{\bar G}_x$.  The same definition applies to ${}^w \alg{\bar T}'$, and minisotropic maximal tori correspond in $\alg{\bar G}_x$ and $\alg{\bar G}_x'$.  Given a minisotropic ${}^w \alg{\bar T}'$, we are interested in characters of the $\F_q$-points,
$$\theta' \From {}^w \bar T' \To \C^\times.$$
We say that $\theta'$ is genuine if it restricts to the character $\epsilon'$ on $\F_q^\times$.  We say that $\theta'$ is in {\em general position} if it is not fixed by any nontrivial element of $W^{\Fr_w}$ (cf.~\cite[Definition 5.15]{DL}). 

From the data of $w$ and $\theta'$, Deligne and Lusztig construct a virtual representation $R( {}^w \bar T', \theta')$ of $\bar G_x'$ such that:
\begin{itemize}
\item If ${}^w \alg{\bar T'}$ is minisotropic, and $\theta'$ is in general position, then for some choice of sign, $\pm R({}^w \bar T', \theta')$ is a cuspidal irrep of $\bar G_x'$.
\item If $\alg{\bar G}_x'$ has connected center (equivalently, if $\alg{G}$ has connected center), these $\pm R({}^w \bar T', \theta')$ are exactly the cuspidal irreps of $\bar G_x'$ which are generic (cf.~\cite[Remark 6.2.7]{DR}).
\end{itemize}
When ${}^w \alg{\bar T'}$ is minisotropic and $\theta'$ is in general position, write $\rho({}^w \bar T', \theta') = \pm R({}^w \bar T', \theta')$ for the resulting cuspidal irrep of $\bar G_x'$.  

If $\theta_1', \theta_2'$ are two characters of ${}^w \bar T'$ in general position, then the following are equivalent (cf.~\cite[Theorem 6.8]{DL})
\begin{enumerate}
\item
The cuspidal irreps are isomorphic:  $\rho({}^w \bar T', \theta_1') \isom \rho({}^w \bar T', \theta_2')$.
\item
The characters $\theta_1'$ and $\theta_2'$ are {\em geometrically conjugate} (see \cite[Definition 5.5]{DL}), written $\theta_1' \sim \theta_2'$.
\end{enumerate}

\begin{prop} \label{P:prep}
Suppose that ${}^w \alg{\bar T}'$ is minisotropic, $\theta'$ is genuine and in general position, and $\rho = \rho({}^w \bar T', \theta')$.  Then 
$$Y_{x, \rho} = \{ y \in Y^{W \rtimes \Fr} : \theta' \sim \theta' \cdot (\epsilon' \circ \bar \chi_y) \},$$
where $\bar \chi_y$ is the character in Corollary \ref{cor-key}, restricted to ${}^w \bar T$. 
\end{prop}
\begin{proof}
This follows directly from Corollary \ref{cor-key}, and the fact that if $\chi \From \bar G_x \To \C^\times$ is a character, then
$$\rho({}^w \bar T', \theta') \cdot \chi \isom \rho({}^w \bar T', \theta' \cdot \chi).$$
This can be shown directly, e.g., by tracing the twist by $\chi$ through the character formula of \cite[Theorem 4.2]{DL}.
\end{proof}

Since we have proven that $\dim \Wh_\psi(\pi_{x, \tau}) = [Y^{W \rtimes \Fr} : Y_{x,\rho}]$, the above proposition gives a tool to compute the dimension of the space of Whittaker functionals.  This requires an analysis of when $\theta'$ is geometrically conjugate to character $\theta \cdot (\eta' \circ \bar \chi_y)$.  Geometric conjugacy is a bit easier to see on the dual side, via Lusztig parameters, and this should be closely related to the Langlands parameterization.

\subsection{Lusztig parameters}

Consider the previous setting, where ${}^w \alg{\bar T}'$ is minisotropic, and $\theta' \From {}^w \bar T' \To \C^\times$ is genuine and in general position.  Associated to this data, Deligne and Lusztig give a semisimple conjugacy class in a dual group, which we review here, following \cite[\S 5]{DL} and \cite[\S 16]{Lus}.  In particular, Lusztig outlines a parameterization using the complex dual group, whereas the previous work of Deligne and Lusztig uses a dual group defined over $\F_q$.  The complex dual group seems more relevant to the Langlands conjectures for covering groups of \cite{We4}.

Choose an isomorphism $\eta \From \bar \F_q^\times \xrightarrow{\sim} (\Q / \Z)_{p'}$ (the prime-to-$p$ subgroup of $\Q / \Z$).  The exponential map $x \mapsto e^{2 \pi i x}$ gives an isomorphism from $(\Q / \Z)_{p'}$ to the group $\mu_{p'}(\C)$ of prime-to-$p$ roots of unity in $\C$.

Recall that $\epsilon \From \mu_n(\F_q) \into \C^\times$ has been chosen, and $\epsilon' \From \F_q^\times \To \C^\times$ is given by $\epsilon'(x) = \epsilon(x^{ (q-1) / n} )$.  For convenience, we assume $\eta$ has been chosen compatibly with $\epsilon$, in the sense that
$$\epsilon(x) = e^{2 \pi i \eta(x) } \text{ for all } x \in \mu_n(\F_q).$$

As ${}^w \alg{\bar T}'$ fits into a short exact sequence $\alg{\bar G}_m \into {}^w \alg{\bar T}' \onto {}^w \alg{\bar T}$, the cocharacter lattice of ${}^w \alg{\bar T}'$ fits into a short exact sequence $\Z \into Y' \onto Y$.  We identify ${}^w \alg{\bar T}'(\bar \F_q) = Y' \otimes \bar \F_q^\times$, and the action of $\Fr$ on ${}^w \alg{\bar T}'(\bar \F_q)$ corresponds to $(\Fr_w \otimes \Fr)$ on $Y' \otimes \bar \F_q^\times$.  But this is the same as the action of $q \Fr_w \otimes \Id$ on $Y' \otimes \bar \F_q^\times$.  Thus we write
$${}^w \bar T' = {}^w \alg{\bar T}'(\F_q) = (Y' \otimes \bar \F_q^\times)^{q \Fr_w}.$$
Using $\eta \From \bar \F_q^\times \xrightarrow{\sim} (\Q / \Z)_{p'}$, we have an identification,
$$\eta \From {}^w \bar T' = \left( Y' \otimes \bar \F_q^\times \right)^{q \Fr_w} \xrightarrow{\sim} \left( Y' \otimes (\Q / \Z)_{p'} \right)^{q \Fr_w}.$$

The complex dual torus is defined by $T'^\vee = X' \otimes \C^\times$.  Following Lusztig \cite[\S 16]{Lus}, we may use the complex dual torus to parameterize characters of ${}^w \bar T'$.  Begin with the pairing,
$$(X' \otimes \Q) \otimes (Y' \otimes \Q) \To \Q, \quad (x,y) \mapsto \langle x, q \Fr_w y - y \rangle.$$
This gives a perfect pairing (see \cite[Eqn.~(5.2.3)*]{DL}),
$$\left( X' \otimes (\Q/\Z)_{p'} \right)^{q \Fr_w} \otimes \left( Y' \otimes (\Q/\Z)_{p'} \right)^{q \Fr_w} \To (\Q/\Z)_{p'}.$$
Applying the exponential map and $\eta$, one finds a perfect pairing,
$$(\bullet, \bullet)_{\Lus} \From \left( X' \otimes \mu_{p'}(\C) \right)^{q \Fr_w} \otimes \left( Y' \otimes \bar \F_q^\times \right)^{q \Fr_w} \To \mu_{p'}(\C).$$
This gives an isomorphism,
$$\Lus \From \Hom( {}^w \bar T', \C^\times) = \Hom(  {}^w \bar T', \mu_{p'}(\C) ) \xrightarrow{\sim} \left( X' \otimes \mu_{p'}(\C) \right)^{q \Fr_w}.$$
Since $X' \otimes \mu_{p'}(\C) \subset X' \otimes \C^\times = T'^\vee$, every character $\theta' \From {}^w \bar T' \To \C^\times$ corresponds to an element $\theta'^\vee = \Lus(\theta) \in T'^\vee$ satisfying
\begin{equation}
\label{thetaq}
(\theta'^\vee)^q = \Fr_w^{-1}(\theta'^\vee).
\end{equation}

Following \cite[Proposition 5.22]{DL}, geometric conjugacy of characters corresponds to conjugacy in the dual group.
\begin{prop}
Two characters $\theta_1', \theta_2' \From {}^w \bar T' \To \C^\times$ are geometrically conjugate if and only if the elements $\theta'^\vee_1, \theta'^\vee_2 \in T'^\vee$ are $W$-conjugate.
\end{prop}

If $y \in Y^{W \rtimes \Fr}$, the character $\epsilon' \circ \bar \chi_y$ can also be shifted to the side of the dual group.  Recall that $y$ gives an element
$$\kappa_y \in \Hom_{\Z[\Fr_w]}(Y/ Y_{\SC}, \Z), \quad \kappa_y(y') = B_Q(y,y').$$
Viewing $\kappa_y$ as an element of $X$, define
$$\xi_y = \kappa_y \otimes e^{2 \pi i / n } \in X \otimes \mu_{p'}(\C).$$
Since $\kappa_y \in X^{\Fr_w}$, and $n \mid (q-1)$, we find that $\xi_y \in (X \otimes \mu_{p'}(\C))^{q \Fr_w}$.

\begin{lm}
The character $\epsilon' \circ \bar \chi_y \From \bar G_x \To \mu_n$, restricted to ${}^w \bar T$, satisfies $\Lus(\epsilon' \circ \bar \chi_y) = \xi_y \in \hat T$.
\end{lm}
\begin{proof}
Suppose that $y' \in Y'$ and $\bar z \in \bar \F_q^\times$.  Then we compute,
\begin{align*}
\left( \xi_y, (y' \otimes \bar z) \right)_\Lus  &= \left( \kappa_y \otimes e^{2 \pi i / n}, y' \otimes \bar z \right)_\Lus \\
&= e^{\frac{2 \pi i }{n} \cdot \eta(\bar z) \cdot  \langle \kappa_y, q \Fr_w y' - y' \rangle } \\
&= e^{\frac{2 \pi i }{n} \cdot \eta(\bar z) \cdot  \langle \kappa_y, (q-1) y'  \rangle }, \quad \text{ (since $\kappa_y$ is $\Fr_w$-invariant)} \\
&= e^{\frac{ 2 \pi i (q-1)}{n} \cdot \eta(\bar z) \cdot B_Q(y,y') },  \\
&= \epsilon' \left( \bar z^{B_Q(y,y')} \right) = \epsilon' \left(\chi_y( y'(\bar z) ) \right).
\end{align*}
\end{proof}

\begin{prop}
\label{thetaProp}
Suppose that $\theta' \From {}^w \bar T' \To \C^\times$ is a genuine character in general position, and $\rho = \rho({}^w T', \theta')$.  Then
$$Y_{x, \rho} = \{ y \in Y^{W \rtimes \Fr} : \theta'^\vee \text{ is $W$-conjugate to } \theta'^\vee \cdot \xi_y \}.$$
\end{prop}

Note here that we are multiplying $\theta'^\vee \in T'^\vee = X' \otimes \C^\times$ and $\xi_y \in \hat T = X \otimes \C^\times$.  For this, we are applying the embedding $X \into X'$ dual to the projection $Y' \onto Y$.
%\section{Examples}
%In the following two subsections, we will consider coverings of $\alg{GL}_r$ and $\alg{GSp}_{2r}$.  Our precise results for these groups indicates that one cannot do much better than the bounds above in general.  In particular, the dimension of $\Wh_\psi(\pi)$ often varies as $\pi$ varies, even for genuine depth zero generic supercuspidal irreps of a fixed group.
%

%%% section %%%
\section{Coverings of $\alg{GL}_r$}

In this section, we will consider coverings of $\alg{GL}_r$.  Our precise results indicate that the dimension of $\Wh_\psi(\pi)$ varies as $\pi$ varies, even within the class of genuine depth zero generic supercuspidal irreps of a fixed covering group.  We hope that these dimensions can be predicted in a future theory of L-packets for covering groups, and this should be seen as a first step in this direction.

In \cite{Blo}, Blondel studied the Whittaker models of supercuspidal representations of the Kazhdan-Patterson covers of $\alg{GL}_r$.  Fundamentally, our methods are similar to those of \cite{Blo}.  But here we obtain more general formulae for Brylinski-Deligne covering groups, and by using Lusztig parameters, the computations become simpler.  

\subsection{Classification of coverings}

Let $\alg{G} = \alg{GL}_r$, with the standard Borel subgroup $\alg{B}$ of upper-triangular matrices and torus $\alg{T}$ of diagonal matrices.  Let $\set{e_1, e_2, ..., e_r}$ be the basis for the cocharacter lattice $Y$ of $\alg{T}$, corresponding to the diagonal entries.  Let $e_0 = \sum e_i$, so $e_0$ generates the lattice $Y^W$ of central cocharacters.  The set of simple coroots is $\Delta^\vee=\set{\alpha_i^\vee:=e_i - e_{i+1}}_{1\le i\le r-1}$, and these span the coroot lattice $Y_{\SC}$.

By the main theorem of Brylinski and Deligne \cite{BD}, central extensions $\alg{K}_2 \into \alg{G}' \onto \alg{G}$ are classified (up to unique isomorphism) by triples $(Q, \sheaf{D}, f)$, where $Q \From Y \To \Z$ is a Weyl-invariant quadratic form and $F^\times \into \sheaf{D} \onto Y$ is a central extension.  The quadratic form $Q$ determines a central extension $F^\times \into \sheaf{D}_Q \onto Y_{\SC}$ (see \cite[\S 11]{BD}), and the ``third invariant'' is an embedding $f \From \sheaf{D}_Q \To \sheaf{D}$ lying over $Y_{\SC} \To Y$.  Since $Y / Y_{\SC}$ is free (i.e., since the derived subgroup of $\alg{G}$ is simply-connected), one can see that the {\em isomorphism class} of such a central extension $\alg{G}'$ is uniquely determined by just the first invariant $Q$. By \cite{We3}, this classification holds for central extensions of $\alg{GL}_r$ by $\alg{K}_2$ over $\O$ just as it holds over $F$.

Weyl-invariant quadratic forms $Q \From Y \To \Z$ are uniquely determined by two integers:
$$\p = Q(e_i) \text{ (for all $1 \leq i \leq r$)} \text { and } \q = B_Q(e_i, e_j) \text{ (for all $1 \leq i < j \leq r$)}.$$
Note that 
$$Q(\alpha^\vee) = Q(e_i - e_j) = Q(e_i) + Q(e_j) - B_Q(e_i, e_j) = 2 \p - \q.$$ 
This number $2 \p - \q$ determines the resulting central extension of $\alg{SL}_r$ by $\alg{K}_2$, and therefore plays an important role in understanding the covers of $\alg{GL}_r$.  That is, if $2 \p - \q = 1$, then the resulting central extension of $\alg{SL}_r$ is Matsumoto's canonical extension.

An easy computation yields
$$Q(e_0) = r \p + \left( {r \atop 2} \right) \q.$$
The number $Q(e_0)$ determines the central extension of the center $\alg{Z} \subset \alg{G}$ by $\alg{K}_2$, up to isomorphism.  Note that, for every $r \geq 1$, the pair of integers $(Q(e_0), Q(\alpha^\vee))$ determines the pair $(\p, \q)$, and vice versa.  

Below we highlight three classes of covers of $\alg{GL}_r$.

%% subsection %%
\subsubsection{Determinantal coverings}
Suppose that $\alg{G}'$ is an extension of $\alg{G} = \alg{GL}_r$ by $\alg{K}_2$, for which 
$$2 \p - \q = 0.$$
Then the pullback of the extension to $\alg{SL}_r$ splits uniquely, $\alg{SL}_r' = \alg{SL}_r \times \alg{K}_2$.  It follows that $\alg{G}'$ can be realized canonically as the pullback of an extension of $\alg{GL}_1$ via the determinant.
$$\begin{tikzcd}
\alg{K}_2 \inarrow{r} \arrow{d}{=} & \alg{G}' \onarrow{r} \arrow{d} & \alg{GL}_r \arrow{d}{\det} \\
\alg{K}_2 \inarrow{r} & \alg{GL}_1' \onarrow{r} & \alg{GL}_1
\end{tikzcd}$$

The isomorphism class of the extension $\alg{GL}_1$ is determined by the single integer $\p$.  More concretely, the extension $\alg{GL}_1'$ is isomorphic to the extension whose underlying sheaf of sets is $\alg{GL}_1 \times \alg{K}_2$, with multiplication given by
$$(u_1, \zeta_1) \cdot (u_2, \zeta_2) = \left( u_1 u_2, \zeta_1 \zeta_2 \cdot \{ u_1, u_2 \}^{\p} \right).$$

Thus the extension $\alg{G}'$ is isomorphic to the extension whose underlying sheaf of sets is $\alg{GL}_r \times \alg{K}_2$, with multiplication given by
$$(g_1, \zeta_1) \cdot (g_2, \zeta_2) = \left( g_1 g_2, \zeta_1 \zeta_2 \cdot \{ \det g_1, \det g_2 \}^{\p} \right).$$
We call these extensions the {\em determinantal coverings} of $\alg{GL}_r$.  All are defined over $\O$.

%% subsection %%
\subsubsection{The Kazhdan-Patterson coverings}
Now suppose that $\alg{G}'$ is an extension of $\alg{G} = \alg{GL}_r$ by $\alg{K}_2$, for which 
$$2 \p - \q = -1.$$
Then the pullback of the extension to $\alg{SL}_r$ is the opposite of Matsumoto's universal central extension, i.e., $Q(\alpha^\vee) = -1$ for every simple coroot (cf.~\cite[Proposition 4.15]{BD}).  In this case, the $n$-fold covering groups $\tilde G$ are exactly those studied by Kazhdan-Patterson \cite[\S 0.1]{KP}. The parameter $\p$ corresponds to the twisting parameter $c$ in the notation of \cite{KP}.  This family is the most widely studied among all Brylinski-Deligne extensions of $\alg{GL}_r$.  When $c = \p = 0$, the covering groups $\tilde G$ are also the focus of earlier works of \cite{Fli}, \cite{GHPS}, etc.

The simplest construction of such an extension $\alg{G}'$ is the following:  let $f \From \alg{GL}_r \To \alg{SL}_{r+1}$ be the block-diagonal embedding given by $g\mapsto (g, \det(g)^{-1})$. Let $\alg{SL}'_{r+1}$ be the (unique up to unique isomorphism) extension such that $Q(\alpha^\vee)=-1$ for any coroot $\alpha^\vee$.  Let $\alg{G}'$ be the pullback of $\alg{SL}_{r+1}'$ to $\alg{GL}_r$ via $f$.  Then the invariants of $\alg{G}'$ are easily computed:  $\p = \q = -1$.  

To construct the other Kazhdan-Patterson extensions, one may simply twist this covering by a suitable determinantal covering.  Here twisting refers to the Baer sum of central extensions, which corresponds to addition of quadratic forms, and thus to addition of invariants $(\p,\q)$.  At the level of cocycles, twisting corresponds to multiplication of cocycles, which is how Kazhdan and Patterson construct their extensions $\widetilde{GL}_r^{(c)}$.   

Since the Matsumoto cover of $\alg{SL}_{r+1}$ can be defined over $\O$, so too can the Kazhdan-Patterson covering with $\p = \q = -1$.  By taking Baer sums of Kazhdan-Patterson coverings ($a$ times) and determinantal coverings ($b$ times), we can find $\O$-models of coverings with invariants 
$$(\p, \q) = a (-1,-1) + b(1,2) \text{ for all } a,b \in \Z.$$
This constructs an $\O$-model of every central extensions of $\alg{GL}_r$ by $\alg{K}_2$.

%% subsection %%
\subsubsection{Savin's nice coverings}
Gordan Savin has recently (see \cite{Sav}) introduced a class of coverings of $GL_r$ which also can be considered in the Brylinski-Deligne category.  He studies extensions $\alg{K}_2 \into \alg{G}' \onto \alg{G} = \alg{GL}_r$ which have invariants $\p, \q$ satisfying
$$2 \p - \q = -2.$$

The simplest construction of such an extension $\alg{G}'$ is the following:  let $h \From \alg{GL}_r \To \alg{Sp}_{2r}$ be the embedding of the Siegel Levi subgroup.  Let $\alg{Sp}_{2r}'$ be the extension of $\alg{Sp}_{2r}$ by $\alg{K}_2$ whose quadratic form $Q$ satisfies $Q(\alpha^\vee) = -1$ for every short coroot.  Let $\alg{G}'$ be the pullback of $\alg{Sp}_{2r}'$ via $h$.  Then the invariants of $\alg{G}'$ are $\p = -1, \q = 0$.

The resulting $n$-fold covering groups $\tilde G$ are particularly nice, because their restrictions to Levi subgroups are ``block-commutative'' -- this is a consequence of the identity $\q = 0$.  All other extensions satisfying $2 \p - \q = -2$ can be obtained from this $\alg{G}'$ by twisting by a determinantal covering. 

%% subsection %%

\subsection{Generic depth zero supercuspidals}
Let $\alg{G}'$ be an extension of $\alg{G} = \alg{GL}_r$ by $\alg{K}_2$ with invariants $\p, \q$, defined over $\O$.  If $n \mid (q-1)$ as usual, we find an $n$-fold cover $\mu_n \into \tilde G \onto G$.  This is obtained as the pushout of the full tame extension $\F_q^\times \into G' \onto G$.  We study the Whittaker functionals of $\epsilon$-genuine generic depth zero supercuspidal irreps of $\tilde G$, or equivalently, $\epsilon'$-genuine irreps of $G'$.

In the building of $G = \alg{GL}_r(F)$, all hyperspecial points are $G$-conjugate, and thus in the orbit of the hyperspecial point $\circ$ arising from the usual $\O$-model of $\alg{GL}_r$.  Thus (up to isomorphism) every genuine depth zero generic supercuspidal irrep of $\tilde G$ has the form $\pi_{\circ, \tau}$, where $\tau$ is a genuine cuspidal irrep of $N_\circ' = Z' G_\circ'$.  As we have fixed an $\O$-model of the extension $\alg{G}'$, we have a distinguished splitting of $G_\circ = \alg{GL}_r(\O)$ into $G'$.  Thus we can write
$$N_\circ' = Z' \cdot G_\circ \text{ and } G_\circ' = \F_q^\times \times G_\circ.$$

The residual extension $\alg{\bar G}_m \into \alg{\bar G}_\circ' \onto \alg{\bar G}_\circ$ inherits a splitting from the $\O$-model of $\alg{G}'$, so
$$\alg{\bar G}_\circ' = \alg{\bar G}_\circ \times \alg{\bar G}_m = \alg{\overline{GL}}_r \times \alg{\bar G}_m.$$
The minisotropic tori in $\alg{\overline{GL}}_r$ are parameterized by the $r$-cycles in the Weyl group $W = S_r$.  Thus there is a unique conjugacy class of minisotropic tori, and we choose $w = (1 2 \ldots r)$ for the $r$-cycle, and ${}^w \alg{\bar T}$ for the resulting minisotropic torus.  There is an identification,
$${}^w \alg{\bar T}(\F_q) \isom \F_{q^r}^\times.$$

The splitting of the residual extension gives a splitting of tori, ${}^w \bar T' = {}^w \bar T \times \F_q^\times$, and so every genuine character of ${}^w \bar T'$ has the form $\theta' = \theta \otimes \epsilon'$ for some
$$\theta \From {}^w T(\F_q) = \F_{q^r}^\times \To \C^\times.$$

If $\theta$ is in general position, we obtain a genuine generic cuspidal irrep $\rho = \rho({}^w T', \theta \otimes \epsilon')$ of $\bar G_\circ'$.  This gives a family of genuine generic depth zero representations $\pi_\tau$, as $\tau$ ranges over irreps of $N_\circ'$ whose restriction contains $\rho$.

In this case, the dual torus $T^\vee = X \otimes \C^\times$ can be identified with the diagonal complex matrices in $GL_r(\C)$.  We write $f_1, \ldots, f_r$ for the $\Z$-basis of $X$ dual to the basis $e_1, \ldots, e_r$ of $Y$.  Define $f_0 = f_1 + \cdots + f_r$.  We view the Lusztig parameter of $\theta$ as a diagonal complex matrix,
$$\theta^\vee = \diag( \theta^\vee_1, \ldots, \theta^\vee_r ), \quad \theta^\vee_i \in \mu_{p'}(\C),$$
and by \eqref{thetaq}, we have
$$\theta_{i+1}^\vee = (\theta_i^\vee)^q \text{ for } 1 \leq i \leq r-1, \quad \theta^\vee_1 = (\theta^\vee_r)^q.$$
Thus the parameter $\theta^\vee$ is determined by the single complex number
$$\theta^\vee_1 \in \C^\times, \quad (\theta^\vee_1)^{q^r} = \theta^\vee_1.$$

For $\theta$ to be in general position, it is necessary and sufficient that $\theta^\vee$ have trivial stabilizer in $W = S_r$.  Equivalently, the numbers $\theta^\vee_1, \ldots, \theta^\vee_r$ are distinct, or equivalently,
$$(\theta_1^\vee)^{q^s} \neq \theta^\vee_1 \text{ for all $1 \leq s < r$.}$$
  
\subsection{Whittaker models}

Let $\theta$ be a character of ${}^w T$ in general position, and $\pi$ a genuine generic depth zero supercuspidal representation of $\tilde G$ compactly induced from $\tau$, an irrep of $N_\circ'$ whose restriction to $\bar G_\circ'$ contains $\rho = \rho({}^w T', \theta \otimes \epsilon')$.

To study $\dim \Wh_\psi(\pi)$, we note that $Y^W = Y^{W \rtimes \Fr} = \Z e_0$, and recall Proposition \ref{thetaProp},
$$Y_{\circ, \rho} = \{ y \in Y^{W \rtimes \Fr} : \theta^\vee \text{ is $W$-conjugate to } \theta^\vee \cdot \xi_y \}.$$
For $1 \leq i \leq r$, we compute
$$\kappa_{e_0}(e_i) = B_Q(e_0, e_i) = 2 \p + (r-1) \q.$$
As shorthand, define $\mm_{Q,r} = 2\p+(r-1)\q$.  Then we have,
\begin{align*}
\xi_{e_0} &= \mm_{Q,r} f_0 \otimes e^{2 \pi i / n}  \\
&= \diag \left( e^{2 \pi i \mm_{Q,r} / n}, \ldots, e^{2 \pi i \mm_{Q,r} / n} \right) \in T^\vee = X \otimes \C^\times.
\end{align*}

Combining the results above, we find a concrete formula for the dimension of the space of Whittaker functionals.
\begin{prop}
The dimension of $\Wh_\psi(\pi)$ is determined by the parameter $\theta^\vee$, according to the formula
$$\dim \Wh_\psi(\pi) = \min \{ k : k > 0 \text{ and } e^{2 \pi i \mm_{Q,r} k / n} \cdot \theta^\vee_1 = (\theta_1^\vee)^{q^s} \text{ for some } s \}.$$
In particular,
$$\dim \Wh_\psi(\pi) \text{ divides } \frac{n}{\gcd(n, \mm_{Q,r})}.$$
\end{prop}

This gives uniqueness of Whittaker models in some important cases, also found by Blondel \cite[\S 3.4(3)]{Blo}.
\begin{cor}
If $n \mid \mm_{Q,r} = 2 \p + (r-1) \q$, then $\dim \Wh_\psi(\pi) = 1$.
\end{cor}

\subsection{Remarks on parameters}

If one accepts a natural local Langlands parameterization for covering groups, as described in \cite{We4}, then the work of DeBacker and Reeder suggests that a depth zero supercuspidal representation of an unramified covering group $\tilde G$ should have a Weil parameter $\phi \From \Weil_F \To {}^\EL \tilde G$.  Here $\tilde G^\vee \into {}^\EL \tilde G \onto \mathrm{Gal}_F$ denotes the L-group of the covering group.  This parameter $\phi$ should send a topological generator of tame inertia to a regular semisimple element $\tilde \theta^\vee \in \tilde G^\vee$, and should send Frobenius to an element of ${}^\EL \tilde G$ acting on the torus via a Weyl element $w^\vee$.  It is straightforward to adapt \cite{DR} to construct $(\tilde \theta^\vee, w^\vee)$ from the Lusztig parameter described earlier, and thus give a conjectural parameter $\phi$.

More difficult is the interpretation of $\dim \Wh_\psi(\pi)$ from the parameter $\phi$.  At a formal, computational level, $\dim \Wh_\psi(\pi)$ can be recovered from $\phi$ (when $\pi$ is known to be generic).  But it seems likely that a full understanding will require one to understand pure (or rigid) inner forms for covering groups.  We leave such a deep study for a future paper.

% -------------------------------------------0---------------------
\bibliographystyle{amsalpha}
\bibliography{WDZ}
% ----------------------------------------------------------------

\end{document}